\documentclass[review,onefignum,onetabnum]{siamart190516}

\usepackage{lipsum}
\usepackage{amsfonts}
\usepackage{graphicx}
\usepackage{epstopdf}
\usepackage{algorithmic}

\usepackage{tcolorbox}
\usepackage{color}

\ifpdf
  \DeclareGraphicsExtensions{.eps,.pdf,.png,.jpg}
\else
  \DeclareGraphicsExtensions{.eps}
\fi

\newsiamremark{hypothesis}{Hypothesis}
\crefname{hypothesis}{Hypothesis}{Hypotheses}
\newsiamthm{claim}{Claim}

\title{Efficient Shallow Ritz Method for Semi-linear Diffusion-Reaction and Allen-Cahn Equations} 

\author{Zhiqiang Cai\thanks{Department of Mathematics, Great Bay University, Dongguan, Guangdong 523000, China (\email{caiz@gbu.edu.cn}).}
\and César Herrera\thanks{Department of Mathematics, Purdue University, West Lafayette, IN 47907, USA (\email{herre125@purdue.edu}).}}

\usepackage{amsopn}

\ifpdf
\hypersetup{
  pdftitle={fastNNSolver3},
  pdfauthor={}
}
\fi

\usepackage{amsmath,rotating}
\DeclareMathOperator*{\argmin}{arg\,min}

\usepackage[utf8]{inputenc}
\usepackage{bm}
\usepackage{bbm}
\usepackage{mathtools}
\usepackage{indentfirst}
\usepackage{subfigure}
\usepackage{tcolorbox}
\usepackage{color}
\usepackage[export]{adjustbox}
\usepackage{scalerel}
\usepackage{tikz}
\usepackage{geometry}
\usepackage{multicol}
\usetikzlibrary{decorations.pathreplacing}
\usepackage{graphicx}
\usepackage{subcaption}
\usepackage{amsmath,amssymb}
\usepackage{nicematrix}
\usepackage{float}

\usepackage{algorithmic}
\Crefname{ALC@unique}{Line}{Lines}

\newcommand{\vertiii}[1]{{\left\vert\kern-0.25ex\left\vert\kern-0.25ex\left\vert #1 
    \right\vert\kern-0.25ex\right\vert\kern-0.25ex\right\vert}}

\makeatletter
\let\orig@nomath\@nomath
\let\orig@scriptsize\scriptsize
\renewcommand{\scriptsize}{%
  \ifmmode
    \let\@nomath\@gobble 
    \orig@scriptsize
    \let\@nomath\orig@nomath 
  \else
    \orig@scriptsize
  \fi}
\makeatother

\newcommand{\R}{\mathbb{R}}

\usepackage{graphicx}
\usepackage{diagbox}
\usepackage{pifont}
\definecolor{RED}{rgb}{1, 0, 0} 
\usepackage{soul}

\newcommand{\bSigma}{\mbox{\boldmath${\Sigma}$}}

\newcommand{\btheta}{\mbox{\boldmath${\theta}$}}

\newcommand{\bxi}{\mbox{\boldmath$\xi$}}

\usepackage{enumitem}
\setlist[itemize]{left=16pt} 

\def\bb{{\bf b}}
\def\bB{{\bf B}}

\def\bc{{\bf c}}
\def\bd{{\bf d}}
\def\bD{{\bf D}}
\def\bH{{\bf H}}

\def\bh{{\bf h}}
\def\bff{{\bf f}}

\def\bv{{\bf v}}

\def\bw{{\bf w}}

\def\bs{{\bf s}}
\def\bg{{\bf g}}
\def\bq{{\bf q}}

\def\cM{{\cal M}}

\Crefname{section}{Section}{Sections}
\crefname{section}{Section}{Sections}

\begin{document}


%

\maketitle
\begin{abstract} 
An efficient shallow Ritz method was developed in \cite{CaiDokFalHer2024a, CaiDokFalHer2024b} for one-dimensional linear diffusion and diffusion-reaction problems. In this paper, the method is extended to a nonlinear diffusion-reaction equation with a cubic reaction term: (1) a basic error estimate of the shallow Ritz discretization is established and (2) a reduced block Newton (rBN) method for solving the resulting non-convex optimization problem is developed. Again, the computational cost of the rBN method is $\mathcal{O}(n)$ for one-dimensional problems, where $n$ is the number of neurons. Combining a convex splitting scheme for time discretization with the efficient shallow Ritz method, we develop a time-stepping neural network method for solving the Allen–Cahn equation. The unconditional energy stability of the time discretization and an error estimate in the $L^2$ norm are established. 

Numerical experiments demonstrate that the neural network approach yields accurate solutions for singularly perturbed nonlinear diffusion-reaction equations and the Allen–Cahn equation. The method effectively resolves interior and boundary layers without oscillations and significantly outperforms the finite element method on fixed uniform meshes.

\end{abstract}
\begin{keywords}
    Fast iterative solvers, Neural network, ReLU activation, Ritz formulation, Diffusion-reaction problems, Allen–Cahn equation, Newton's method
\end{keywords}

\section{Introduction}

Neural networks (NNs) have achieved remarkable success in artificial intelligence, natural language processing, and image recognition, and their applications have increasingly extended to scientific computing. In particular, NN-based methods have emerged as promising alternatives for the numerical solution of partial differential equations (PDEs) \cite{BERG201828, CAI2020109707, Do19, E_Yu_2018, RAISSI2019686, sirignano2018dgm}, with early developments dating back to the 1990s \cite{Dissanayake94, Lagaris00, Lagaris98}. One of the main advantages of NN approximations is their ability to adapt to the underlying solution structure through the optimization of network parameters. This feature is particularly attractive for problems exhibiting sharp transitions, boundary layers, or interior layers, where conventional numerical methods on uniform meshes may require a large number of degrees of freedom to achieve satisfactory accuracy.

In one dimension, shallow neural networks with rectified linear unit (ReLU) activation functions admit representations equivalent to free-knot splines (FKS). Unlike conventional splines with prescribed knots, FKS allow the knot locations to be optimized along with the coefficients, potentially achieving substantially improved approximation accuracy with relatively few degrees of freedom (see \cite{Burchard74} and the discussion in \cite{CaiDokFalHer2024a}). From this perspective, the optimization of the hidden-layer parameters of a shallow ReLU network corresponds to the adaptive placement of spline knots, while the output-layer parameters determine the coefficients of the spline representation. However, this flexibility comes at the cost of a challenging non-convex optimization problem. Furthermore, the global support of ReLU basis functions generally leads to dense and ill-conditioned algebraic systems \cite{Cai24GN}. Consequently, the development of efficient and robust optimization algorithms is essential for realizing the full potential of shallow NN approximations in the numerical solution of PDEs.

Recently, an efficient shallow Ritz method was developed in \cite{CaiDokFalHer2024a, CaiDokFalHer2024b} for one-dimensional linear diffusion and diffusion-reaction problems. A key ingredient of this approach is a block Newton (BN) method that exploits the distinct roles of the linear and nonlinear network parameters. Following a classical outer-inner iterative strategy, the method alternates between optimizing the output-layer coefficients and updating the hidden-layer parameters, or equivalently, the free-knot locations. Such a separation of linear and nonlinear parameters is natural and has been explored in various contexts \cite{Ainsworth2020PlateauPI, Ainsworth2022, pmlr-v107-cyr20a, Golub1973, separable3, park2022neuron, separable2}. Nevertheless, two major computational difficulties remain: the optimization of the linear parameters involves dense, ill-conditioned algebraic systems, while the optimization of the nonlinear parameters requires solving nonlinear equations whose Hessian matrices may be singular. By exploiting the special structure of the shallow ReLU representation, the methods developed in \cite{CaiDokFalHer2024a, CaiDokFalHer2024b} effectively address these difficulties and reduce the computational cost to $\mathcal{O}(n)$ in one dimension, where $n$ denotes the number of neurons. In addition, a local convergence analysis of the BN method was recently established in \cite{cai_conv}, providing convergence conditions that can be incorporated into practical implementations.

The objective of this paper is to extend the efficient shallow Ritz method from linear elliptic problems to semi-linear diffusion-reaction equations with cubic nonlinearities and, subsequently, to develop an efficient NN-based time-stepping method for the Allen–Cahn equation. The extension to semi-linear problems introduces several new analytical and computational challenges. First, although the Ritz functional remains convex with respect to the output-layer coefficients for fixed hidden-layer parameters, the corresponding first-order optimality conditions now form a system of nonlinear equations rather than a linear system. Second, the Hessian matrices associated with the optimization of both the linear and nonlinear parameters depend on both sets of parameters, making their efficient evaluation and inversion more involved. Nevertheless, by extending the structural identities and fast inversion techniques developed in \cite{CaiDokFalHer2024a, CaiDokFalHer2024b}, we construct a reduced block Newton (rBN) method that retains the $\mathcal{O}(n)$ computational complexity for one-dimensional problems. We also establish a basic error estimate for the shallow Ritz discretization of the semi-linear diffusion-reaction equation.

An important application of the proposed method is the Allen–Cahn equation, originally introduced in \cite{ALLEN19791085} to describe the motion of antiphase boundaries in crystalline solids. The Allen–Cahn equation has since become a fundamental phase-field model with applications in materials science, interface dynamics, and fluid mechanics (see, e.g., \cite{deteix2022new,geng2022new, li2023second,li2021unconditionally,yang2021fast}). In the singularly perturbed regime, its solutions typically exhibit narrow transition layers separating regions corresponding to different stable phases. Accurately resolving these layers is challenging for conventional numerical methods, particularly on fixed uniform meshes, where insufficient spatial resolution may lead to significant errors and spurious oscillations. The adaptive nature of shallow ReLU networks makes them especially suitable for such problems, since the free-knot locations can be concentrated near sharp transitions without requiring a large number of neurons. Indeed, the numerical results in \cite{CaiDokFalHer2024b} have already demonstrated the effectiveness of the shallow Ritz method for singularly perturbed linear elliptic problems. These observations provide a strong motivation for extending the method to the time-dependent, nonlinear Allen–Cahn equation.

For the temporal discretization of the Allen–Cahn equation, convex splitting schemes have been widely studied because of their favorable stability properties \cite{doi:10.1137/0730084, guan2014second, feng2015long, shen2010numerical}. In particular, the first-order convex splitting scheme introduced in \cite{doi:10.1137/0730084} is unconditionally energy stable, preserving a discrete analogue of the energy dissipation property of the continuous problem without imposing a restriction on the time-step size. At each time step, the scheme requires the solution of a semi-linear elliptic problem with a cubic reaction term. This structure provides a natural connection between the convex splitting time discretization and the proposed shallow Ritz method. By combining the two approaches, we develop an efficient time-stepping NN method for the Allen–Cahn equation, in which the semi-linear elliptic problem at each time level is solved using the rBN algorithm. We establish the unconditional energy stability of the convex splitting time discretization and derive an error estimate in the $L^2$ norm.

The main contributions of this paper are summarized as follows:

\begin{itemize}
    \item[1.] {\bf Shallow Ritz discretization and error analysis for semi-linear problems.} We extend the shallow Ritz formulation to one-dimensional semi-linear diffusion-reaction equations with cubic reaction terms and establish a basic error estimate for the resulting NN approximation.
    \item[2.] {\bf Efficient reduced block Newton method.} We develop an rBN algorithm for solving the non-convex optimization problem arising from the shallow Ritz discretization. By exploiting the special algebraic structure of shallow ReLU networks, the method achieves $\mathcal{O}(n)$ computational complexity in one dimension, despite the additional nonlinearities in the first-order optimality conditions and Hessian matrices.
    \item[3.] {\bf Energy-stable NN time-stepping method for the Allen–Cahn equation.} We combine the efficient shallow Ritz method with a first-order convex splitting scheme to obtain a time-stepping NN method for the Allen–Cahn equation. Unconditional energy stability of the time discretization and an $L^2$-norm error estimate are established.
    \item[4.] {\bf Accurate resolution of sharp transition layers.} Numerical experiments demonstrate that the proposed method accurately resolves interior and boundary layers in singularly perturbed semi-linear diffusion-reaction problems and the Allen–Cahn equation. Compared with the finite element method on fixed uniform meshes, the shallow Ritz method achieves substantially higher accuracy with relatively few degrees of freedom and avoids spurious oscillations near sharp transitions.
\end{itemize}

The remainder of this paper is organized as follows. In \cref{sec1}, we introduce the shallow Ritz discretization for semi-linear diffusion-reaction problems and establish its basic error estimate. In \cref{sec2}, we derive the algebraic systems associated with the first-order optimality conditions. In \cref{sec3}, we develop the reduced block Newton method and discuss its implementation and local convergence conditions. In \cref{sec4}, we present the convex splitting time discretization for the Allen–Cahn equation, establish its energy stability and error estimate, and describe the resulting NN time-stepping method. Numerical experiments illustrating the accuracy, efficiency, and robustness of the proposed approach are reported in \cref{Section numerical exp}.

\section{Shallow Ritz Method for Semi-Linear Diffusion-Reaction Problems} \label{sec1}
Consider the following semi-linear diffusion-reaction equation in one dimension:
\begin{equation}\label{pde}
    \left\{\begin{array}{lr}
        -(a(x)u^{\prime}(x))^{\prime} +r_1(x)u(x)^3 + r_2(x)u(x)=f(x), & x \in I,\\[2mm]
        u(0)= \alpha,  \quad (a u')(1)=\beta, &
    \end{array}\right.
\end{equation}
where the diffusion coefficient $a(x)$, the reaction coefficients $r_1(x)$, $r_2(x)$, 
and the right-hand side function $f(x)$ are given real-valued functions defined on $I$. 
Assume that $a(x),\, r_1(x),\, r_2(x)$ are in $L^\infty(I)$ and bounded below by constants $\mu_0 > 0$, $\mu_{1} \geq 0$, and $\mu_{2} \geq 0$, respectively, almost everywhere on $I$. 
Unless otherwise stated, for simplicity of notation we take $I = (0,1)$; the method and analysis extend directly to any open bounded interval. Here, we consider an inhomogeneous Neumann boundary condition at $x = 1$; for the Dirichlet case, we refer the reader to \cite{CaiDokFalHer2024a, CaiDokFalHer2024b}.

The Ritz formulation of problem \cref{pde} is to find $u \in H^{1}(I)\cap \{u(0)= \alpha\}$ such that
\begin{equation}\label{energy_functional}
J(u) = \min_{v \in H^{1}(I)\cap \{v(0)= \alpha\}}J(v),
\end{equation}
with the energy functional 
\begin{equation*}
J(v) = \frac{1}{2}\int_{0}^1\left(a(x)v'(x)^2 
     + \frac{1}{2}r_1(x)v(x)^4 
     + r_2(x)v(x)^2\right)dx - f(v),
\end{equation*}
where $v'(x)^2 = (v'(x))^2$ and the linear functional is given by
\begin{equation*}
    f(v) = \int_{0}^1 f(x)\,v(x)\,dx + \beta\,v(1).
\end{equation*}

The set of approximating functions generated by the shallow ReLU neural network with $n$ interior breakpoints over the domain $I$ is denoted by 
\begin{eqnarray*}\label{ShallowNN}
    {\cal M}_n(I)&=&\left\{c_{-1}+ c_0x+  \sum_{i=1}^{n}c_i\sigma(x-b_i) \, :\, c_i\in \R,\, 0 = b_0 <  b_1<\cdots<b_n<b_{n+1}=1\right\},  
\end{eqnarray*}
where $\sigma(t) =\max\{0, t\}$ is the ReLU activation function. Let 
\[
{\cal M}^\alpha_n(I)={\cal M}_n(I) \cap \{u_n(0) = \alpha \}.
\]
Then, the Ritz neural network approximation is to find $u_n\in {\cal M}^\alpha_n(I)$ such that 
\begin{equation}\label{min_pde}
    J(u_n)=\min_{v\in\cM^\alpha_n(I)}J(v).
\end{equation}

Next, we establish an error bound similar to C\'{e}a's lemma for the solution $u_n$ of \cref{min_pde} in the weighted $H^1(I)$ norm defined by 
\begin{equation*}
    \vertiii{v}^2_{a,r_2} = \int_0^1 \left(a(x)v'(x)^2 + r_2(x)v(x)^2 \right) dx.
\end{equation*}
This bound is for a nonlinear approximation defined over a subset but not a subspace of $H^1(I)$. 

\begin{lemma}\label{l:error_estimate1}
Let $u$ and $u_n$ be the solutions of problems \cref{energy_functional} and \cref{min_pde}, respectively. Then 
\begin{equation}\label{LeaL}
\vertiii{u - u_n}_{a,r_2}
\le \inf_{v \in {\cal M}^\alpha_n(I)} E(u-v;a, r_1,r_2), 
\end{equation}
where the square of $E(u-v;a, r_1,r_2)$ is given by
\begin{equation*}
E^2(u-v;a, r_1,r_2) = \vertiii{u-v}^2_{a,r_2}    + \|r_1^{1/4}(u-v)\|_{L^4(I)}^4    + 5\|r_1^{1/4}u\|_{L^4(I)}^2         \|r_1^{1/4}(u-v)\|_{L^4(I)}^2.    
\end{equation*}
\end{lemma}

\begin{proof}
First, we show that for any $v \in {\cal M}^\alpha_n(I)$, we have
\begin{equation}\label{ineq1}
    \vertiii{u - v}^2_{a,r_2} \leq 2\big(J(v) - J(u)\big),
\end{equation}
as in the linear case (see the proof of Theorem 3.2 in \cite{LiuCai2}). To this end, multiplying \cref{pde} by $w \in \{u, v\}$ and using integration by parts, we have
\begin{equation*}
f(w) = \int_0^1 \left(a\,u'w' + r_1\,u^3 w + r_2\,u w\right)dx + \alpha a(0)u'(0),
\end{equation*}
which implies
\begin{equation*}
f(v)-f(u) = \int_0^1 \left(a\,u'(v'-u') + r_1\,u^3 (v-u) + r_2\,u (v-u)\right)dx.
\end{equation*}
A straightforward calculation gives 
\begin{equation}\label{jv-ju}
   2(J(v) - J(u)) =\vertiii{u-v}^2_{a,r_2} + \dfrac12 \int_0^1 r_1\biggl( v^4 + 4u^3(u-v) - u^4\biggr)\,dx, 
\end{equation}
which, together with the identity 
\begin{equation}\label{iden}
    v^4 + 4u^3(u-v) - u^4 = (u-v)^2\left[(u+v)^2+2u^2 \right]\ge 0,
\end{equation}
implies the validity of \cref{ineq1}.
To bound $J(v)-J(u)$, first we have
\[
(u+v)^2+2u^2 \leq  2(u-v)^2 + 10 u^2 ,
\]
which, together with \cref{iden} and the Cauchy–Schwarz inequality, implies
\[
\int_0^1 r_1\biggl( v^4 + 4u^3(u-v) - u^4\biggr)\,dx \leq 2\|r_1^{1/4}(u-v)\|_{L^4(I)}^4 + 10 \|r_1^{1/4}u\|_{L^4(I)}^2 \|r_1^{1/4}(u-v)\|_{L^4(I)}^2.
\]
Now, \cref{LeaL} follows from \cref{ineq1} with $v=u_n$, the fact that $J(u_n) \leq J(v)$ for all $v \in {\cal M}^\alpha_n(I)$, \cref{jv-ju}, and the above inequality. This completes the proof of the lemma.
\end{proof}

\begin{remark}\label{r1}
    The above proof is independent of dimension, and hence the lemma holds in multiple dimensions.
\end{remark}

\section{Optimality Conditions}\label{sec2}This section derives systems of algebraic equations arising from the optimality conditions of \cref{min_pde}.

To this end, for any function $v(x)\in {\cal M}^\alpha_{n}(I)$, let
\[
v(x)= \alpha+  \sum_{i=0}^{n}c_i\sigma(x-b_i).
\]
Denote by 
\[
\bc = (c_0, \dots, c_n)^T \quad\mbox{and}\quad \bb =(b_1, \dots, b_n)^T
\]
the respective linear and nonlinear parameters, and denote by $\btheta = (\bc, \bb) \in \mathbb{R}^{2n+1}$ the vector containing all optimizable parameters. If $u_n(x)\in {\cal M}^\alpha_{n}(I)$ is a solution of \cref{min_pde}, then $F(\btheta) = J(u_n(\cdot; \btheta))$, as a function of all parameters, satisfies 
the following first-order optimality conditions
\begin{equation}\label{cps_ode}
 \nabla_{\scriptsize\btheta} F\left(\btheta \right) = \begin{pmatrix}  \nabla_{\bc} F\left(\btheta\right)
 \\[2mm] \nabla_{\bb} F\left(\btheta\right)\end{pmatrix} 
 = \begin{pmatrix}
 {\bf 0} \\[2mm] {\bf 0}\end{pmatrix},
\end{equation}
 where $\nabla_{\bc}$ and $\nabla_{\bb}$ denote the gradients with respect to $\bc$ and $\bb$, respectively. 

Denote by $H(t)$ and $\delta (t)$ the respective Heaviside (unit) step and Dirac delta functions. It is known that 
they are the respective first and second {\it weak} derivatives of $\sigma(t)$, i.e., $H(t)=\sigma'(t)$ and $\delta(t)=\sigma''(t)$.
For each $i = 0,1, \dots, n$, let
\begin{equation*}
    \sigma_i(x)=\sigma(x-b_i), \quad  H_i(x)=H(x-b_i), \quad \mbox{and} \quad \delta_{i}(x)=\delta(x-b_i).
\end{equation*}
Set
\begin{equation*}
    {\bf \Sigma}_{n+1}(x)=\begin{pmatrix}
        \sigma_0(x)\\\sigma_1(x)\\
        \vdots\\
        \sigma_n(x)
    \end{pmatrix},\: \bH_n(x)=\begin{pmatrix}H_1(x)\\
    \vdots\\H_n(x)       
    \end{pmatrix},\: \bH_{n+1}(x) = \begin{pmatrix}
        H_{0}(x)\\
        H_1(x)\\
        \vdots\\
        H_{n}(x)
    \end{pmatrix},\mbox{ and } {\bf \Lambda}_n(x)=\begin{pmatrix}
        \delta_{1}(x)\\
        \vdots \\
        \delta_n(x)
    \end{pmatrix}.
\end{equation*} 
Let $\hat{r}(x;\btheta) = r_1 u_n ^3 + r_2u_n$. From the facts that
\begin{equation*}
u_n(x)\in H^1(I), \,\, \nabla_{\bc} u'_n(x) = \bH_{n+1}(x),\,\, \nabla_{\bc} u_n(x) = {\bf \Sigma}_{n+1}(x), \mbox{ and } \nabla_{\bc} u_n (1)  = {\bf \Sigma}_{n+1}(1)=\bd,
\end{equation*}
the first block in \cref{cps_ode} becomes
    \begin{align} \notag 
    {\bf 0} = \nabla_{\bc}F(\btheta) & = \int_0^1 \left[a u_n'\nabla_{\bc} u'_n + \hat{r}(x;\btheta) \nabla_{\bc} u_n \right] dx - \int_0^1 f \nabla_{\bc} u_n  dx -\beta \nabla_{\bc} u_n(1) \\[2mm] \label{linear_eq}
    &= 
    A_{n+1}(\bb;a) \bc + \int_{0}^1\hat{r}(x;\btheta)\bSigma_{n+1} dx  - 
    \bff_{n+1}(\bb) -\beta\, \bd,
\end{align}
where the matrix $A_{n+1}(\bb;a)$ and the vector $\bff_{n+1}(\bb)$ are given by
\begin{equation*}
    A_{n+1}(\bb;a)=\int_0^1 a(x)\bH_{n+1}(x)\bH_{n+1}(x)^T dx
    \quad \mbox{and} \quad \bff_{n+1}(\bb)=\int_0^1 f(x)\bSigma_{n+1}(x)dx,
\end{equation*}
respectively.

Next, to derive the second block in \cref{cps_ode}, for $j=1,\ldots,n$, notice first that (see the proof of Lemma 4.1 in \cite{CaiDokFalHer2024a})
\begin{equation}\label{grad-b-dif}
   \dfrac{\partial}{\partial b_j}\int_{0}^{1}a(x)u^{\prime}_n(x)^2dx 
     = -2c_ja(b_j) u_n'(b_j),
\end{equation}
where $u_n'(b_j) = \displaystyle\sum_{i=0}^{j-1}c_i+ \frac{c_j}{2}$ due to the convention that $H(0)=1/2$. Let
\begin{equation}\label{q}
q_j =q_j(\btheta)= \int_{b_j}^1\bigl[f(x) - \hat{r}(x;\btheta)\bigr]dx  - a(b_j)u_n'(b_j). 
\end{equation}
Since 
\[
\dfrac{\partial u_n(x)}{\partial b_j}=-c_jH(x - b_j) \quad\mbox{and}\quad \dfrac{\partial u_n(1)}{\partial b_j}=-c_j,
\]
it follows from \cref{grad-b-dif} that the second block in \cref{cps_ode} becomes
    \begin{align} \notag 
    {\bf 0} = \nabla_{\bb}F(\btheta) & = \dfrac12 \nabla_{\bb} \int_0^1 a(x) u'_n(x)^2dx +  \int_0^1 (\hat{r}(x;\btheta) -f)\nabla_{\bb} u_n dx -\beta \nabla_{\bb} u_n(1) \\[2mm] \label{opt_b_pde}
    &= \bD_n(\hat{\bc}) \left(\bq  + \beta\, {\bf 1}\right),
\end{align}
where $\bD_n(\hat{\bc})$ is an $n\times n$ diagonal matrix with the $i^{th}$ diagonal element $c_i$ and 
\[
\hat{\bc}= (c_1,\ldots, c_n)^T, \quad  \bq = (q_1, \ldots, q_n)^T, \quad\mbox{and}\quad {\bf 1} = (1, \ldots, 1)^T.
\]

\begin{remark}\label{c0_rmk}
     If $c_i = 0$, then the $i^{th}$ equation is absent from the optimality condition \cref{opt_b_pde}. Moreover, the $i^{th}$ neuron does not contribute to the approximation $u_n$. Such neurons can be removed or redistributed. 
\end{remark}

\begin{remark}\label{linearization}
Equation \cref{linear_eq} is no longer a system of linear algebraic equations due to the nonlinearity of the term $u_n(x)^3$. A possible linearization strategy is as follows: for a given $\hat{u}_n(x) \in {\cal M}^\alpha_n(I)$ such that $\hat{u}_n(x) \approx u_n(x)$, we have
\[
u_n(x)^3 = \bigl(u_n(x)^3 - u_n(x)\bigr) + u_n(x) 
\approx \bigl(\hat{u}_n(x)^3 - \hat{u}_n(x)\bigr) + u_n(x).
\]
This strategy was employed in the initialization for the numerical experiments in \cref{Section numerical exp}.
\end{remark}

\section{A  Block Newton Method}\label{sec3}
This section outlines a block Newton (BN) method for solving the system of nonlinear algebraic equations given in \cref{cps_ode} with $\nabla_{\bc}F(\btheta)$ and $\nabla_{\bb}F(\btheta)$ given in the respective \cref{linear_eq} and \cref{opt_b_pde}.

\subsection{Hessian Matrix}

The Hessian matrix of the function $F(\btheta) = J(u_n(\cdot; \btheta))$ is denoted by
\begin{equation*}
\nabla^2_{\scriptsize\btheta}F(\btheta) 
= {\cal H}(\btheta)= \begin{pmatrix} {\cal H}_{11}(\btheta) & {\cal H}_{12}(\btheta)\\[2mm] 
                 {\cal H}_{21}(\btheta) & {\cal H}_{22}(\btheta)  
\end{pmatrix}
\in \mathbb{R}^{(2n+1) \times (2n+1)},
\end{equation*}
where the block matrices are given by
\begin{equation*}\label{defOperatos}
    {\cal H}_{11}(\btheta) = \nabla^2_{\bc\bc}F(\btheta), \,\,
    {\cal H}_{12}(\btheta) = \nabla^2_{\bc \bb}F(\btheta), \,\,
    {\cal H}_{21}(\btheta) = \nabla^2_{\bb \bc}F(\btheta), \,\mbox{ and }\,
    {\cal H}_{22}(\btheta) = \nabla^2_{\bb\bb}F(\btheta).
\end{equation*}

First, we derive the principal blocks of the Hessian matrix. Denote the weighted mass matrix by 
\[
  M_{n+1}(\bb; r) = \displaystyle\int_{0}^1r(x;\btheta){\bf \Sigma}_{n+1}{\bf \Sigma}_{n+1}^Tdx,
\]
where the weight is $r(x)  =r(x;\btheta) =  3r_1(x)u_n(x)^2 + r_2(x)$. Clearly, we have
\[
\nabla_{\bc} (A_{n+1}(\bb;a) \bc)^T = A_{n+1}(\bb;a) \quad\mbox{and}\quad \nabla_{\bc}\int_{0}^1\Bigl[r_1 u_n ^3 + r_2u_n\Bigr]\bSigma_{n+1}^T dx = M_{n+1}(\bb; r),
\]
which, together with \cref{linear_eq}, gives the $1\times 1$ block of the Hessian matrix
\begin{equation}\label{hessianc}
        {\cal H}_{11}(\btheta)  = A_{n+1}(\bb;a) + M_{n+1}(\bb; r).
    \end{equation}
For $i, j = 1, \ldots, n$, by \cref{q}, a direct calculation yields
    \[
    \frac{\partial q_j}{\partial b_i} = 
    \begin{cases}
        \displaystyle c_i\int_0^1 r(x;\btheta) H_i(x)^2 \, dx + g_i, & \quad i = j, \\[4mm]
        \displaystyle c_i\int_0^1 r(x;\btheta) H_i(x)H_j(x) \, dx, & \quad i \neq j,
    \end{cases}
    \]
where $g_i$ is given by 
\begin{equation}\label{gi}
 g_i = g_i(\btheta) = \hat{r}(b_i;\btheta) - f(b_i) - a^\prime(b_i)u_n'(b_i).
\end{equation} 
Denote by $\bD_n(\bg) = \text{diag}(g_1, \dots, g_n)$ the diagonal matrix with entries $g_1, \dots, g_n$ on the diagonal and by $A_{n}(\bb;r) = \displaystyle\int_{0}^1r(x;\btheta)\bH_n\bH_n^Tdx$ the stiffness matrix. Then the $2\times 2$ block of the Hessian matrix has the form 
    \begin{equation}\label{hessianB}
         {\cal H}_{22}(\btheta)  = \bD_n(\hat{\bc})\big[   A_{n}(\bb;r) \bD_n(\hat{\bc}) + \bD_n(\bg) \big].
    \end{equation}

Next, we derive the off-diagonal block ${\cal H}_{12}(\btheta)$ of the Hessian matrix. For $i = 0,1, \ldots, n$ and $j = 1, \ldots, n$, by \cref{q} we have
\begin{align*}
   \frac{\partial q_j}{\partial c_i} 
   = - \int_{b_j}^1 r\sigma_i(x)\,dx - a(b_j)H_i(b_j),
\end{align*}
which, together with  \cref{opt_b_pde}, implies
\begin{align*}
    \frac{\partial^2 F(\btheta)}{\partial c_i \partial b_j}  
    &= \delta_{ij}\big(q_j(\btheta)+\beta \big) - c_j \left(\int_{b_j}^1 r(x)\sigma_i(x)\,dx + a(b_j)H_i(b_j) \right) \\[2mm]
    &= \delta_{ij}\big(q_j(\btheta)+\beta \big) - c_j \left(\int_{0}^1 r(x)\sigma_i(x)H_j(x)\,dx + \int_{0}^1 a(x)H_i(x) \delta_j(x)dx\right),
\end{align*}
where $\delta_{ij}$ is the Kronecker delta symbol. Then the off-diagonal block ${\cal H}_{12}(\btheta)$ has the form 
\begin{equation}\label{H12Steady}
    {\cal H}_{12}(\btheta) = {\cal R}(\btheta) 
    - \left[C(\bb;a)
    + B(\bb;r) \right]\bD_n(\hat{\bc}),
\end{equation}
where ${\cal R}(\btheta)=\bigl({\cal R}_{ij}(\btheta)\bigr)_{(n+1)\times n}$ with ${\cal R}_{0j}(\btheta)=0$ and ${\cal R}_{ij}(\btheta)=(q_i(\btheta) + \beta)\delta_{ij}$ for $i=1,\ldots,n$; and the matrices $B(\bb;r)$ and $C(\bb;a)$ are given by
\[
B(\bb;r) = \displaystyle\int_{0}^1r(x; \btheta){\bf \Sigma}_{n+1}\bH_{n}^Tdx \quad\mbox{and}\quad C(\bb;a) = \displaystyle\int_{0}^1a(x)\bH_{n+1} {\bf\Lambda}_n^{T}dx.
\]

\smallskip

Define the following $2\times 2$ block matrices
\begin{align*}
& \mathbb{H}_{\Sigma}(\btheta) = 
\begin{pmatrix}
    M_{n+1}(\bb; r)     & -B(\bb;r) \\[4mm]
         - B^T(\bb;r) &    A_n(\bb;r) 
    \end{pmatrix}, \quad 
\mathbb{H}_{\Lambda}(\btheta)=\begin{pmatrix}  A_{n+1}(\bb;a) & -C(\bb;a) \\[4mm]
                 -C^T(\bb;a)  & {\bf 0}_{n\times n} \end{pmatrix}, \\[4mm]
&{\cal D}(\hat{\bc}) = \begin{pmatrix} I_{n+1} & {\bf 0}\\[4mm] {\bf 0} & \bD_n(\hat{\bc})\end{pmatrix}, \, {\cal D}_0(\btheta) = \begin{pmatrix} {\bf 0} & {\bf 0}\\[4mm] {\bf 0} & \bD_n(\hat{\bc})\bD_n(\bg)\end{pmatrix}, \,\mbox{ and }\, \mathbb{R}(\btheta) = \begin{pmatrix}
    {\bf 0} & {\cal R}(\btheta) \\[4mm]
    {\cal R}^T(\btheta) & {\bf 0}
\end{pmatrix}.
\end{align*}
Then, by \cref{hessianc}, \cref{hessianB}, and \cref{H12Steady}, the Hessian matrix at $\btheta$ has the form
\begin{equation*}
    \nabla^2_{\scriptsize{\btheta}}F(\btheta) = \mathbb{R}(\btheta) +
    {\cal D}(\hat{\bc}) \mathbb{H}_{\Sigma+ \Lambda}(\btheta) {\cal D}(\hat{\bc}) + {\cal D}_0(\btheta),
\end{equation*}
where $\mathbb{H}_{\Sigma+\Lambda}(\btheta) = \mathbb{H}_{\Sigma}(\btheta) + \mathbb{H}_{ \Lambda}(\btheta)$.

\subsection{Matrix Inversion}

The $1\times 1 $ block of the Hessian matrix, ${\cal H}_{11}(\btheta)$ given in \cref{hessianc}, is symmetric and positive definite (see \cite{CaiDokFalHer2024b}). This subsection discusses the invertibility of the $2\times 2$ block of the Hessian matrix, ${\cal H}_{22}(\btheta)$ given in \cref{hessianB}, as well as how to efficiently compute their inverses. 

To examine the invertibility of the matrix ${\cal H}_{22}(\btheta)$, let $\tilde{h}_1=h_1$ and let
\begin{equation*}
    I_i=[b_i,b_{i+1}], \quad h_i = b_{i+1} - b_{i}, 
    \quad \text{and}\quad
    \tilde{h}_i = \min\{h_{i-1}, h_i\}. 
\end{equation*}
Clearly, the matrix ${\cal H}_{22}(\btheta)$ is symmetric. The following lemma provides a sufficient condition for its positive definiteness. 

\begin{lemma}\label{sufff}
For all $i=1,\ldots,n$, assume that $a(x)$ is differentiable at $b_i$, that $c_i \neq 0$, and that
    \begin{equation}\label{inv_cond}
        \frac{g_i}{c_i}+\left( 3\mu_1\min_{x \in [b_{i-1}, b_{i+1}]}u_n^2(x) + \mu_2\right)\frac{\tilde{h}_i}{4}>0, 
    \end{equation}
then ${\cal H}_{22}(\btheta)$ is positive definite.
\end{lemma}

\begin{proof}
By the assumptions that $r_i(x)\ge \mu_i\ge 0$  for almost all $x\in I$ ($i=1,\,2$), we have 
\begin{equation}\label{r}
r(x) = 3r_1(x)u^2_n(x) + r_2(x) \ge 3\mu_1\min_{x \in [b_{i-1}, b_{i+1}]}u_n^2(x) + \mu_2, \quad\mbox{for almost all }\, x\in [b_{i-1}, b_{i+1}].  
\end{equation}
Let $\hat{r}_i= 3\mu_1\min\limits_{x \in [b_{i-1}, b_{i+1}]}u_n^2(x) + \mu_2$.
For any vector $\bxi =(\xi_1,\ldots,\xi_n)^T\in \mathbb{R}^{n}$, let 
\[
v(x)=\bxi^T \bD_n(\hat{\bc})\bH_n(x) = \sum\limits_{j=1}^n c_j\xi_j H_j(x).
\]
It follows from the triangle inequality, the definitions of $\tilde{h}_i$ and $\hat{r}_i$, the fact that $v(x)$ is piecewise constant, and \cref{r} that 
\begin{align*}\notag
    \dfrac12 \sum_{i=1}^n \tilde{h}_i\hat{r}_i\Bigl(c_i\xi_i\Bigr)^2 & \leq \sum_{i=1}^n {h}_i\hat{r}_i\left(\sum_{j=1}^i c_j\xi_j \right)^2 + \sum_{i=2}^n {h}_{i-1}\hat{r}_i\left(\sum_{j=1}^{i-1} c_j\xi_j \right)^2 \\[2mm]
    &= \sum_{i=1}^n \hat{r}_i \int_{b_i}^{b_{i+1}}v^2(x)\,dx + \sum_{i=2}^n \hat{r}_i\int_{b_{i-1}}^{b_{i}}v^2(x)\,dx \leq 2 \int^1_0 r(x) v^2(x)\,dx,
\end{align*}
which, together with \cref{hessianB}, implies 
$$\bxi^T {\cal H}_{22}(\btheta)\bxi = \int^1_0 r(x) v^2(x)\,dx + \sum\limits_{i=1}^n c_ig_i \xi^2_i \geq  \sum\limits_{i=1}^n c_i^2\left[\frac{\tilde{h}_i \hat{r}_i}{4} + \dfrac{g_i}{c_i}\right] \xi_i^2.$$ 
Thus, the positive definiteness of ${\cal H}_{22}(\btheta)$ is a direct consequence of the assumptions. This completes the proof of the lemma.
\end{proof}


Below, we discuss the three assumptions in \cref{sufff}. To this end, let 
\begin{equation}\label{setS1}
    S_1 = \bigl\{i \in \{1, \dots, n\} : |c_i| < \tau_1 \quad \text{or} \quad b_i \notin I \bigr\} \,\,\mbox{ and }\,\, S_2 = \bigl\{i \in \{1, \dots, n\} : a'(b_i) \text{ DNE} \bigr\},
\end{equation}
where $\tau_1 \geq 0$ is a small tolerance threshold. $S_1$ and $S_2$ are the sets of indices for the respective non-contributing and fixed neurons. If $c_l = 0$, then the corresponding neuron does not contribute to the current approximation and hence may be removed or redistributed. If $a(x)$ is not differentiable at $b_l$, then $b_l$ is at an interface and hence should remain fixed without further updates. For more details, see \cite{Cai24GN, CaiDokFalHer2024a, CaiDokFalHer2024b}.

Regarding assumption \cref{inv_cond} of \cref{sufff}, we justify its validity for all $i\not\in S_1\cup S_2$ provided that the approximation $u_n(x) \approx u(x)$ is sufficiently accurate. To this end, by \cref{gi}, we have
\begin{align*}
    g_i & = r_1(b_i)u_n(b_i)^3 + r_2(b_i)u_n(b_i) -f(b_i)- a^\prime(b_i)u_n'(b_i)\\
    &\approx  r_1(b_i) u(b_i)^3 +  r_2(b_i)u(b_i)  -f(b_i)- a^\prime(b_i)u'(b_i) =a(b_i)u''(b_i).
\end{align*}
Since the linear parameter $c_i = u_n'(b_i^+) - u_n'(b_i^-)$ represents the derivative jump of $u_n(x)$ at $b_i$ and approximates $u'(b_i^+) - u'(b_i^-)$, it generally shares the same sign as $u''(b_i)$. Hence, we obtain
\begin{equation*}
\dfrac{g_i}{c_i} \approx \dfrac{a(b_i) |u^{\prime\prime}(b_i)|}{|c_i|} >0,
\end{equation*}
which implies the validity of \cref{inv_cond}.



In the remainder of this subsection, we present factorizations of the main blocks of the Hessian matrix in order to compute their inverses in ${\cal O}(n)$ operations. These factorizations were introduced in Appendix~A.2 of \cite{CaiDokFalHer2024b}. To this end, denote by 
\[
\bD_{n+1}(\bh) = \operatorname{diag}(h_0, \ldots, h_n) \quad\mbox{and}\quad \bD_{n}(\bh) = \operatorname{diag}(h_1, \ldots, h_n)
\]
the diagonal matrices consisting of entries $\{h_i\}$, and denote by 
\[
\bD_{n+1}(\bw)=\operatorname{diag}(w_0,\ldots, w_n), \quad \bD_{n+1}(\bs^k)=\operatorname{diag}(s_0^k, \ldots, s_n^k), \quad\mbox{and}\quad \bD_{n}(\bs^k)=\operatorname{diag}(s_1^k, \ldots, s_n^k)
\]
the diagonal matrices consisting of the respective entries 
\[
w_i=\int_{b_i}^{b_{i+1}}a(x)\,dx \quad\mbox{and}\quad
s_i^k=\int_{b_i}^{b_{i+1}}r(x)(x-b_i)^k\,dx.
\]
Let $G_{j}$ be the $j \times j$ lower bidiagonal matrix with $1$ on the main diagonal and $-1$ on the subdiagonal, and set 
\[
Q_j = G_{j}\bD_{j}(\bh)^{-1}G_{j} \quad\mbox{for }\,\,j = n, n+1.
\]
Then the Hessian matrices admit the following explicit factorizations:
\begin{equation*}
    {\cal H}_{11} (\btheta)= Q_{n+1}^{-T}T_{11}Q_{n+1}^{-1} \quad\mbox{and}\quad {\cal H}_{22}(\btheta)  = \bD_n(\hat{\bc})\,Q_{n}^{-T}T_{22} Q_{n}^{-1}\bD_n(\hat{\bc}) + \bD_n(\hat{\bc})\bD_n(\bg),
\end{equation*}
where $T_{11} = T_1 + T_2$ and $T_{22} = G_{n}^T\bD_n(\bh)^{-2}\bD_{n}(\bs^{0})G_{n}$. Here, $ T_1$ and $T_2$ are given by
\begin{align*}
    T_{1} &= (\mathbf{I}-G_{n+1}^{T})\bD_{n+1}(\bs^0)(\mathbf{I}-G_{n+1}) +(\mathbf{I}-G_{n+1}^{T})\bD_{n+1}(\bh)^{-1}\bD_{n+1}(\bs^1)G_{n+1} \\[2mm]
    &\qquad +G_{n+1}^T \bD_{n+1}(\bh)^{-1}\bD_{n+1}(\bs^1)(\mathbf{I}-G_{n+1}) + G_{n+1}^{T}\bD_{n+1}(\bh)^{-2}\bD_{n+1}(\bs^2)G_{n+1} \\[2mm]
  \mbox{and}\quad  T_2 &= G_{n+1}^T\bD_{n+1}(\bh)^{-2}\bD_{n+1}(\bw)G_{n+1},
\end{align*}
respectively.

\subsection{A Block Newton Method}\label{sec3.3}

Instead of solving the nonlinear system in \cref{cps_ode} using the standard Newton method, we employ a block Newton (BN) method. This approach leverages the physical meanings of the linear (coefficients) and nonlinear (mesh points) parameters, resulting in a lower computational cost per iteration. More importantly, this allows us to remove singularities within the $2\times 2$ block. The BN method employs an outer–inner iteration that alternates between the variables $\bc$ and $\bb$. The nonlinear Gauss–Seidel method is used for the outer iteration, while each block is solved with a Newton iteration in the inner step. More precisely, given the current iterate $\left(\bc^{(k)}, \bb^{(k)}\right)$, the next iterate $\left(\bc^{(k+1)}, \bb^{(k+1)}\right)$ is obtained by solving the decoupled block linear system
\begin{equation}\label{GS}
    \begin{pmatrix}
        {\cal H}_{11}\left(\bc^{(k)}, \bb^{(k)}\right) & 0 \\[2mm]
        0 & {\cal H}_{22}\left(\bc^{(k+1)}, \bb^{(k)}\right)
    \end{pmatrix}
    \begin{pmatrix}
        \Delta \bc^{(k+1)} \\[2mm]
        \Delta \bb^{(k+1)}
    \end{pmatrix}
    = - \begin{pmatrix}
        \nabla_{\bc} F\left(\bc^{(k)}, \bb^{(k)}\right) \\[2mm]
        \nabla_{\bb} F\left(\bc^{(k+1)}, \bb^{(k)}\right)
    \end{pmatrix},
\end{equation}
where $\Delta \bc^{(k+1)} = \bc^{(k+1)} - \bc^{(k)}$ and $\Delta \bb^{(k+1)} = \bb^{(k+1)} - \bb^{(k)}$.
Other outer iterative methods, such as Jacobi or linear Gauss–Seidel, may also be employed; see the discussion in section~2 of \cite{cai_conv}.

Set $\btheta^{(k)} = \left(\bc^{(k)}, \bb^{(k)} \right)$ and $\bar{\btheta}^{(k)} = \left(\bc^{(k+1)}, \bb^{(k)} \right)$, and let 
\[
r^{(k)}(x) = r\left(x; \btheta^{(k)}\right) \quad\mbox{and}\quad \bar{r}^{(k)}(x) = r\left(x; \bar{\btheta}^{(k)}\right).
\]
By \cref{hessianc} and \cref{hessianB}, the system of linear equations in \cref{GS} may be written as
\[
\left\{\begin{array}{rl}
   \biggl(A_{n+1}\left(\bb^{(k)}; a\right) + M_{n+1}\left(\bb^{(k)}; r^{(k)}\right) \biggr) \Delta \bc^{(k+1)}\!\!\!\!  &= -\nabla_{\bc} F\left(\btheta^{(k)}\right), \\[4mm]
   \biggl(\bD_n(\hat{\bc}^{(k+1)})A_{n}\left(\bb^{(k)}; \bar{r}^{(k)}\right)\bD_n(\hat{\bc}^{(k+1)}) + \bD_n(\hat{\bc}^{(k+1)})\bD_n\left(\bar{\bg}^{(k)}\right) \biggr)\Delta \bb^{(k+1)} \!\!\!\! & = - \nabla_{\bb} F\left(\bar{\btheta}^{(k)}\right),
\end{array}\right.
\]
where $\bD_n\left(\bar{\bg}^{(k)}\right) = \text{diag}\left(g_1\left(\bar{\btheta}^{(k)}\right), \dots, g_n\left(\bar{\btheta}^{(k)}\right)\right)$.


In what follows, we outline the implementation of the BN method using a reduced nonlinear system constructed based on the sets defined in \cref{setS1}. Denote by
\begin{equation*}
    S = \{1, \dots, n\} \setminus \left( S_1 \cup S_2\right)
\end{equation*}
the set of indices corresponding to the neurons that contribute to the approximation $u_n$ and do not lie at an interface. Given a vector $\bv \in \mathbb{R}^n$, denote by $\bv_{S} \in \mathbb{R}^{n-\left\lvert S^c\right\rvert}$ the vector obtained by removing the $i^{th}$ entry from $\bv$ for every $i \in S^c=S_1\cup S_2$. Similarly, for a matrix $\bB \in \mathbb{R}^{n\times n}$, define $\bB_{S} \in \mathbb{R}^{(n-\lvert S^c\rvert)\times (n-\lvert S^c\rvert)}$ as the matrix obtained by removing the $i^{th}$ row and column of $\bB$ for each $i \in S^c$.

Define the \textit{reduced} search direction vector for the Newton step as
\begin{equation}\label{reduce_Newt}
      \bd_{R}(\btheta) = -\biggl({\cal  H}_{22}(\btheta)_{S}\biggr)^{-1}\bigg(\nabla_{\bb}F(\btheta)\biggr)_{S}.
\end{equation}
Now we are ready to outline the implementation of the reduced block Newton (rBN) method (see \cref{alg:dBN} for the pseudocode). Given a current iterate $\btheta^{(k)}$, compute the next iterate $\btheta^{(k+1)} = \left(\bc^{(k+1)}, \bb^{(k+1)} \right)$ as follows:
\begin{itemize}
    \item[(i)] \textit{Compute the linear parameters} 
\begin{equation*}
    \bc^{(k+1)} = \bc^{(k)} - \left({\cal H}_{11}\left(\btheta^{(k)}\right)\right)^{-1}\nabla_{\bc}F\left(\btheta^{(k)}\right).
\end{equation*}
\item[(ii)] \textit{Compute the search direction $\mathbf{p}^{(k)}=\left(p_1^{(k)},\dots,p_n^{(k)}\right)^T$} by
    \begin{equation}\label{directionVector}
    \bigl(p^{(k)}_i \bigr)_{i \in S} = \bd_R\left(\bar{\btheta}^{(k)}\right) \quad \mbox{and} \quad \bigl(p^{(k)}_i\bigr)_{i \notin S} = {\bf 0},
    \end{equation}
    where $\bd_R\left(\bar{\btheta}^{(k)}\right)$ is the \textit{reduced} Newton direction vector defined in \cref{reduce_Newt}.
    \item [(iii)] \textit{Compute the nonlinear parameters}
    \begin{equation*}
        \bb^{(k+1)} = \bb^{(k)} +\mathbf{p}^{(k)}.
    \end{equation*}
    \item [(iv)] \textit{Redistribute non-contributing breakpoints} $b_i^{(k+1)}$ for all $i \in S_1$ and sort $\bb^{(k+1)}$.   
\end{itemize}

\begin{algorithm}
    \caption{Reduced block Newton (rBN) method for \cref{min_pde}} \label{alg:dBN} 
    \begin{algorithmic}
    \REQUIRE{Initial network parameters $\bb^{(0)}$}
    \ENSURE{Network parameters $\bc$, $\bb$}
    \FOR{$k=0,1,\ldots$}
    \STATE $\triangleright$ \textit{Linear parameters}
    \STATE{$\bc^{(k+1)}\leftarrow \bc^{(k)} - {\cal H}_{11}\left(\btheta^{(k)}\right)^{-1}\nabla_{\bc}F\left(\btheta^{(k)}\right)$}
    \STATE $\triangleright$
    \textit{Nonlinear parameters}
    \STATE{{\textit{Compute the search direction} }$\mathbf{p}^{(k)}$ \textit{as in}} \cref{directionVector}
    \STATE{$\bb^{(k+1)} \leftarrow \bb^{(k)} + \mathbf{p}^{(k)}$}
    \STATE $\triangleright$ \textit{Redistribute non-contributing neurons and sort $\bb^{(k+1)}$}
\ENDFOR
    \end{algorithmic}
\end{algorithm}

\begin{remark}\label{rmk_redist}For the redistribution used in step (iv) in the numerical experiments reported in \cref{Section numerical exp}, we followed the procedure described in \cite{CaiDokFalHer2024a, CaiDokFalHer2024b}: for a neuron $b_l^{(k+1)}$ with $l \in S_1$, we set 
\[
    b_l^{(k+1)} \gets \frac{b^{(k+1)}_{m-1} + b^{(k+1)}_{m}}{2},
\]
where $m \in \{1, \dots, n+1\}$ is an integer chosen uniformly at random.
\end{remark}


\subsection{Local Convergence}
This subsection establishes a local convergence condition for the BN method applied to the minimization problem in \cref{min_pde}. The condition follows from the following theorem proved in \cite{cai_conv}. Here, we assume that $f(x) \in C(I)$ and that $a(x) \in C^1(I)$.

\begin{theorem}\label{main_lemma}
    If $\nabla^2_{\scriptsize{\btheta}}F(\btheta^{*})$ is symmetric positive definite (SPD), then  
    the BN method described by \cref{GS} converges locally to $\btheta^{*}$ in the norm induced by $\nabla^2_{\scriptsize{\btheta}}F(\btheta^{*})$.
    \end{theorem}

If $c_j^* \neq 0$ for all $j = 1, \ldots, n$, it follows from \cref{opt_b_pde} that ${\cal R}_{jj}(\btheta^*)$ vanishes at the critical point $\btheta^*$. Consequently, ${\cal R}(\btheta^*) = {\bf 0}$, and the Hessian matrix at $\btheta^*$ reduces to the following form:
\begin{equation*}
            \nabla^2_{\scriptsize{\btheta}}F(\btheta^*) =  {\cal D}(\hat{\bc}^*) \mathbb{H}_{\Sigma+ \Lambda}(\btheta^*) {\cal D}(\hat{\bc}^*) + {\cal D}_0(\btheta^*).
        \end{equation*}

Let 
$r_0 =  r_0(\btheta) = \displaystyle\min_{x\in I}\{3 r_1(x)(u_n(x))^2 + r_2(x)\}$. The following theorem is obtained in a similar fashion to Theorem 4.4 in \cite{cai_conv}.


\begin{theorem}\label{thmDR}Let $\btheta^{*} = (\bc^*,\bb^*)^T$ be a minimizer of problem \cref{min_pde}. For all $i =1, \dots, n$, assume that $c^*_i \neq 0$ and 
        \begin{equation*}
            \frac{g_i^*}{c_i^*}+ \dfrac{r_0^*\tilde{h}^*_i}{24} > \dfrac{a^2(b^*_i)}{2\mu_0}\left(\dfrac{1}{h_{i-1}^*} + \dfrac{1}{h_{i}^{*}}\right).
        \end{equation*}
Then the Hessian matrix at the critical point $\btheta^{*}$, $\nabla^2_{\scriptsize{\btheta}} F(\btheta^{*}) =\nabla^2_{\scriptsize{\btheta}}J(u_n(\cdot;\btheta^{*}))$, is SPD. Moreover, the BN method locally converges to $\btheta^{*}$ in the norm induced by $\nabla^2_{\scriptsize{\btheta}}F(\btheta^{*})$.
    \end{theorem}

\section{Allen–Cahn Equation}\label{sec4}Consider the scalar Allen–Cahn equation in one spatial dimension
\begin{equation}\label{AC_PDE}
    u_t(x, t) = -\mathcal{L}(u), \quad (x, t) \in I \times [0, T] 
\end{equation}
with initial and boundary conditions
\begin{equation*}
    \left\{\begin{array}{ll}
        u(x, 0) = u_0(x), & x \in I \quad\mbox{and} \\[2mm]
        u(0, t) = \alpha(t),  \quad \varepsilon^2u_x(1, t) = \beta(t), & t \in [0, T],
    \end{array}\right.
\end{equation*}
respectively, where $\varepsilon> 0$ is a given parameter, $T>0$ is the final time, and the semi-linear operator $\mathcal{L}(v)$ is given by
\[
\mathcal{L}(v)= \mathcal{L}_1(v) + \mathcal{L}_2(v) \quad\mbox{with }\, \mathcal{L}_1(v)= -\varepsilon^2 v_{xx}(x, t) + v(x, t)^3 \,\mbox{ and }\,
\mathcal{L}_2(v)=  - v(x, t). 
\]
Equation \cref{AC_PDE} is the $L^2$ gradient flow of the Ginzburg–Landau free energy \cite{doi:10.1137/0730084} 
\begin{equation*}
    {\cal E}(v) = \frac{1}{2}\int_{0}^1 \varepsilon^2 v'(x)^2 \, dx 
    + \int_0^1 {\cal F}(v(x)) \, dx,
\end{equation*}
where ${\cal F}(v) = \dfrac{1}{4}\bigl(v^2 - 1\bigr)^2$ is the double-well potential. 

\subsection{Convex-Splitting Shallow Ritz Method}\label{CSSR}

Below, we describe a convex-splitting shallow Ritz method for the semi-linear parabolic partial differential equation in \cref{AC_PDE} by adopting  a convex splitting scheme (CSS) (see \cite{doi:10.1137/0730084}) for time discretization and the shallow Ritz approach (see \cref{sec1}) for spatial discretization at each time step. This discretization leads to a non-convex optimization problem at each time step that can be efficiently solved by the block Newton method developed in \cref{sec3.3}. 

The double-well potential ${\cal F}(v)$ admits the following convex splitting 
\[
{\cal F}(v) = {\cal F}_c(v)  - {\cal F}_e(v) \quad\mbox{with }\, {\cal F}_c(v)= \frac{1}{4} \bigl(v^4 + 1\bigr) \,\mbox{ and }\, {\cal F}_e(v)= \frac{1}{2} v^2,
\]
which leads to a convex splitting of the energy functional 
\begin{equation*}
    {\cal E}(v) = {\cal E}_c(v) - {\cal E}_e(v),
\end{equation*}
where ${\cal E}_c(v)$ and ${\cal E}_e(v)$ are convex functionals given by
\begin{equation}\label{splitE}
    {\cal E}_{c}(v) = \frac{1}{2}\int_{0}^1 \varepsilon^2 v'(x)^2 \, dx 
    + \frac{1}{4}\int_0^1 \bigl(v(x)^4 + 1\bigr)\, dx
    \,\mbox{ and }\,
    {\cal E}_{e}(v) = \frac{1}{2}\int_0^1 v(x)^2 \, dx,   
\end{equation}
respectively. 

To describe a convex splitting scheme, partition the temporal interval $[0,T]$ by
\[
0=t_0 < t_1 < \cdots < t_N=T.
\]
Denote by $\Delta t_m = t_m-t_{m-1} > 0$ the time step at time $t_{m}$ and by $\Delta t = \max\limits_{1\leq m \leq N}\Delta t_m$ the largest step size. Let $u^{(m-1)}(x)$ be the approximation of the solution $u(x,t)$ at time $t_{m-1}$. The approximation $u^{(m)}(x)$ at time $t_m$ by the first-order CSS (see, e.g., \cite{doi:10.1137/0730084, shen2010numerical}) is the solution of the following semi-linear diffusion-reaction problem
\begin{equation}\label{csPDE}
    \left\{
    \begin{array}{ll}
        \dfrac{u^{(m)}(x) - u^{(m-1)}(x)}{\Delta t_m}   
    = -\mathcal{L}_1\left(u^{(m)}(x)\right)  - \mathcal{L}_2\left(u^{(m-1)}(x)\right) , &  x \in I, \\[4mm]
      u^{(m)}(0) = \alpha(t_{m}), \quad   {\varepsilon^2} u^{(m)}_x(1) = \beta(t_{m}). &
    \end{array}
    \right.
\end{equation}
This temporal discretization scheme is unconditionally energy stable, i.e., $\mathcal{E}(u^{(m)}) \leq \mathcal{E}(u^{(m-1)})$ for all $m = 1, \ldots, N$.

\begin{remark}
For second-order semi-implicit schemes, one may consider the following second-order BDF scheme:
\begin{equation*}
\dfrac{3u^{(m)} - 4u^{(m-1)} + u^{(m-2)}}{2\Delta t} = \varepsilon^2 u^{(m)}_{xx} - 2\mathcal{F}'(u^{(m-1)}) + \mathcal{F}'(u^{(m-2)}),
\end{equation*}
or the stabilized version presented in {\em \cite{shen2010numerical}}, which can be written in the form of \cref{pde}. However, these schemes are neither unconditionally stable nor unconditionally energy stable.
\end{remark}

Clearly, \cref{csPDE} can be expressed in the form of \cref{pde} with
\begin{equation}\label{dr-tm}
    a(x) = \varepsilon^2, \,\, r_1(x) = 1, \,\, r_2(x) = \tfrac{1}{\Delta t_m}, \,\mbox{ and }\,
f(x) = f^{(m)}(x) = \bigl(1 + \tfrac{1}{\Delta t_m}\bigr)u^{(m-1)}(x).
\end{equation}
Define the energy functional at time $t_m$ by
\begin{equation*}
    J^{(m)}(v)  = \frac{1}{2}\int_{0}^1\left(\varepsilon^2v'(x)^2
    + \frac{1}{2}v(x)^4 + \tfrac{1}{\Delta t_m} v(x)^2\right)dx - f^{(m)}(v),
\end{equation*}
where the linear functional $f^{(m)}(\cdot)$ is given by
\[
f^{(m)}(v) = \bigl(1 + \tfrac{1}{\Delta t_m}\bigr) \int_0^1 u^{(m-1)}(x)\,v(x)\,dx
+ \beta(t_{m})\,v(1).
\]
With ${\cal V}_m = H^{1}(I)\cap \{v(0)= \alpha(t_m)\}$, the minimization formulation
of problem \cref{csPDE} is to find $u^{(m)}(x) \in {\cal V}_m$ such that
\begin{equation}\label{umn}
    u^{(m)}(x) = \argmin_{v \in {\cal V}_m} J^{(m)}(v).
\end{equation}


For spatial discretization, let $u_n^{(0)}(x) \in {\cal M}^{\alpha(0)}_n(I)$ be the following least-squares approximation of the initial data $u_0(x)$:
\begin{equation}\label{u_0LS}
    u_n^{(0)}(x) = \argmin_{v \in {\cal M}^{\alpha(0)}_n(I)} \frac{1}{2}\int_0^1 \bigl(v(x) - u_0(x)\bigr)^2 \, dx 
    = \argmin_{v \in {\cal M}^{\alpha(0)}_n(I)} \frac{1}{2}\big\|v - u_0\big\|^2_{L^2(I)}.
\end{equation}
For computing $u_n^{(0)}(x)$ efficiently, see section~6 in \cite{CaiDokFalHer2024b}. For $m=1,\ldots,N$, define the functional $J_n^{(m)}(\cdot)$ at time $t_m$ by
\begin{equation}\label{functionalACn}
    J_n^{(m)}(v) = \frac{1}{2}\int_{0}^1\left(\varepsilon^2v'(x)^2
    + \frac{1}{2}v(x)^4 + \tfrac{1}{\Delta t_m} v(x)^2\right)dx - f_n^{(m)}(v),
\end{equation}
where $f_n^{(m)}(v)$ is the linear form given by
\[
f_n^{(m)}(v) = \left(1 + \dfrac{1}{\Delta t_m}\right) \int_0^1 u_n^{(m-1)}(x)\,v(x)\,dx
+ \beta(t_{m})\,v(1).
\]
Then the convex-splitting shallow Ritz approximation $u_n^{(m)}(x) \in {\cal M}^{\alpha(t_{m})}_n(I)$ at time $t_m$ is obtained by
minimizing $J_n^{(m)}(\cdot)$ over ${\cal M}^{\alpha(t_{m})}_n(I)$, i.e.,
\begin{equation}\label{minAC}
    u_n^{(m)}(x) = \argmin_{v \in {\cal M}^{\alpha(t_{m})}_n(I)} J_n^{(m)}(v).
\end{equation}

The minimization problem in \cref{minAC} is non-convex and usually has many solutions and/or many local minima; therefore, it is critical to start with a good initial guess in order to obtain the desired solution. Generally, the approximation at the previous time step $u_n^{(m-1)}(x)$ is a good initialization for the current approximation  $u_n^{(m)}(x)$, and so are the corresponding parameters. Pseudocode for the convex-splitting shallow Ritz method is given in \cref{alg:BNAC}.

\begin{algorithm}
    \caption{Convex-splitting shallow Ritz method for the Allen–Cahn equation}\label{alg:BNAC}
    \begin{algorithmic}
    \REQUIRE Initial and boundary conditions $u_0(x)$, $\alpha(t)$, and $\beta(t)$
    \ENSURE Network parameters $\bigl\{(\bc_m, \bb_m)\bigr\}_{m=0}^{N}$
    \STATE $\triangleright$ \textit{Compute the parameters $(\bc_0, \bb_0)$ by solving}
    \STATE{$u_n^{(0)}(x) = u_n(x; \bc_0, \bb_0) = \displaystyle\argmin_{v \in{\cal M}^{\alpha(0)}_n(I)}\|v - u_0\|^2_{L^2(I)}$}
    \FOR{$m = 1, \ldots, N$}
        \STATE  $\triangleright$ \textit{Compute the parameters $(\bc_m, \bb_m)$ by solving the minimization problem \cref{minAC} using the rBN method initialized with $(\bc_{m-1}, \bb_{m-1})$}
    \ENDFOR
    \end{algorithmic}
\end{algorithm}

\subsection{Stability and Error Analysis} 

This subsection analyzes the stability and error of the convex-splitting shallow Ritz method described in \cref{CSSR}. For stability and convergence analysis of other convex splitting schemes, we refer the reader to \cite{guan2017convergence, shen2010numerical} and the references therein.

Let $(\cdot,\cdot)$ and $\|\cdot\|$ denote the $L^2(I)$ inner product and norm, respectively. The next lemma will be used to prove an unconditional stability result.

\begin{lemma}\label{NNenergystab}
Assume that the boundary data are independent of time, $\alpha(t)\equiv\alpha$ and
$\beta(t)\equiv\beta$, and that $u_n^{(m-1)} \in {\cal M}^{\alpha}_n(I)$. 
If $v_n \in {\cal M}^{\alpha}_n(I)$ satisfies
$J_n^{(m)}\bigl(v_n\bigr) \le J_n^{(m)}\bigl(u_n^{(m-1)}\bigr)$, then
\begin{equation}\label{NNdecay}
{\cal E}\bigl(v_n\bigr) - {\cal E}\bigl(u_n^{(m-1)}\bigr)
+ \frac{1}{2}\Bigl(1 + \frac{1}{\Delta t_m}\Bigr)
\bigl\|v_n - u_n^{(m-1)}\bigr\|^2 \leq \beta\Bigl(v_n(1) - u_n^{(m-1)}(1)\Bigr)
\end{equation}
for all $m = 1, \ldots, N$.
\end{lemma}

\begin{proof}
Let 
\[
G_n^{(m)}(v) =\dfrac{1}{2} \left(1+ \frac{1}{\Delta t_m}\right)\left(\int^1_0v^2dx - 2\int^1_0u_n^{(m-1)}vdx \right).
\]
Then \cref{splitE} and \cref{functionalACn} give
\[
J_n^{(m)}\bigl(v\bigr) = {\cal E}(v) + G_n^{(m)}(v) -\dfrac14 -\beta v(1).
\]
Since $G_n^{(m)}(v_n)-G_n^{(m)}\left(u_n^{(m-1)}\right)= \frac{1}{2}\Bigl(1 + \frac{1}{\Delta t_m}\Bigr)
\left\|v_n - u_n^{(m-1)}\right\|^2$, we have
\[
J_n^{(m)}\bigl(v_n\bigr) - J_n^{(m)}\bigl(u_n^{(m-1)}\bigr)
= {\cal E}\bigl(v_n\bigr) - {\cal E}\bigl(u_n^{(m-1)}\bigr)
- \beta\Bigl(v_n(1) - u_n^{(m-1)}(1)\Bigr)
+ \frac{1}{2}\Bigl(1 + \frac{1}{\Delta t_m}\Bigr)
\bigl\|v_n - u_n^{(m-1)}\bigr\|^2,
\]
which, together with the assumption, implies \cref{NNdecay}. This completes the proof of the lemma.
\end{proof}

\begin{theorem}\label{NNenergystabcor}
Under the assumptions of \cref{NNenergystab}, the convex-splitting shallow Ritz approximation
$u_n^{(m)}$ defined in \cref{minAC} satisfies
\begin{equation}\label{NNdecaysum}
    {\cal E}\bigl(u_n^{(N)}\bigr)
    + \frac{1}{2}\sum_{m = 1}^{N}\Bigl(1 + \frac{1}{\Delta t_m}\Bigr)
    \bigl\|u_n^{(m)} - u_n^{(m-1)}\bigr\|^2
    \;\leq\; {\cal E}\bigl(u_n^{(0)}\bigr)
    + \beta\bigl(u_n^{(N)}(1) - u_n^{(0)}(1)\bigr).
\end{equation}
Moreover, if $\beta = 0$, then the method is unconditionally energy stable, i.e.,
\begin{equation}\label{NNenergydecay}
    {\cal E}\bigl(u_n^{(N)}\bigr) \;\leq\; {\cal E}\bigl(u_n^{(0)}\bigr),
\end{equation}
and unconditionally stable in the following measures
\begin{equation}\label{NNstab}
    \varepsilon\,\bigl|u_n^{(N)}\bigr|_{H^1(I)}
    \;\leq\; \sqrt{2\,{\cal E}\bigl(u_n^{(0)}\bigr)}
    \quad\mbox{and}\quad
    \bigl\|u_n^{(N)}\bigr\|_{L^\infty(I)}
    \;\leq\; |\alpha| + \frac{1}{\varepsilon}\sqrt{2\,{\cal E}\bigl(u_n^{(0)}\bigr)} .
\end{equation}
\end{theorem}

\begin{proof}
\cref{NNdecaysum} follows directly from the summation of \cref{NNdecay} with $v_n= u_n^{(m)}$ over $m$ from $1$ to $N$,  
and \cref{NNenergydecay} is a direct consequence of \cref{NNdecaysum} and the assumption that $\beta=0$. The first inequality in \cref{NNstab} follows from \cref{NNenergydecay} and the fact that 
\[
\frac{\varepsilon^2}{2}\,\bigl|v\bigr|_{H^1(I)}^2 \leq {\cal E}(v) \quad\forall\, \, v \in H^1(I).
\]
Finally, using $u_n^{(N)}(0) = \alpha$, the fundamental theorem of calculus, and the Cauchy–Schwarz
inequality, we have 
\[
\bigl|u_n^{(N)}(x)\bigr| \leq |\alpha| + \int_0^x \bigl|\bigl(u_n^{(N)}\bigr)'(s)\bigr|\,ds
\leq |\alpha| + \bigl|u_n^{(N)}\bigr|_{H^1(I)} ,
\]
for all $x \in I$, which, combined with the first inequality in \cref{NNstab}, gives the validity of the second inequality. This completes the proof of the theorem.
\end{proof}

The next lemma shows that the CSS is first-order accurate in time and makes explicit the dependence of the constant on the solution. 

\begin{lemma}\label{temporalerror}
Let $u(x,t)$ and $u^{(m)}(x)$ be the solutions of \cref{AC_PDE} and \cref{umn} with $u^{(0)} = u_0$, respectively.
Assume that $\Delta t\leq 1/4$ and that $u_t(x,t),\, u_{tt}(x,t) \in L^{\infty}(0,T;L^2(I))$. Then
\begin{equation}\label{temporalest}
\left\|u(\cdot,T) - u^{(N)}\right\|
\;\leq\; C_T \Delta t  ,
\end{equation}
where $\Delta t$ is the largest step size and $C_T$ is a constant given by
\[
C_T = \sqrt{2T}\,e^{3T} \left(\|u_t\|_{L^{\infty}(0,T;L^2(I))}
+ \dfrac12\|u_{tt}\|_{L^{\infty}(0,T;L^2(I))}\right).
\]
\end{lemma}

\begin{proof}
 The exact solution $u(x,t)$ at time $t_m$ satisfies the following equation
\begin{equation}\label{trunc}
\frac{u(x,t_m) - u(x,t_{m-1})}{\Delta t_m}
= - \mathcal{L}_1\left(u(x,t_m) \right) - \mathcal{L}_2\left(u(x,t_{m-1})\right) + R^{(m)}(x) ,
\end{equation}
where
\begin{align*}
R^{(m)}(x) &= \frac{u(x,t_m) - u(x,t_{m-1})}{\Delta t_m}
+ \mathcal{L}_1\bigl(u(x,t_m)\bigr) + \mathcal{L}_2\bigl(u(x,t_{m-1})\bigr) \\[2mm]
&= \left(\frac{u(x,t_m) - u(x,t_{m-1})}{\Delta t_m} - u_t(x,t_m)\right)
+ \Bigl(\mathcal{L}_2\bigl(u(x,t_{m-1})\bigr) - \mathcal{L}_2\bigl(u(x,t_m)\bigr)\Bigr) \\[2mm]
&= 
- \frac{1}{\Delta t_m}\int_{t_{m-1}}^{t_m} (t - t_{m-1})\,u_{tt}(x,t)\,dt + \int_{t_{m-1}}^{t_m} u_t(x,t)\,dt ,
\end{align*}
where the second equality uses \cref{AC_PDE} and the third uses Taylor's expansion with an integral remainder.
Therefore,
\begin{equation}\label{remainderbound}
\left\|R^{(m)}\right\| \leq \Delta t_m
\left(\|u_t\|_{L^{\infty}(0,T;L^2(I))}
+ \frac12\|u_{tt}\|_{L^{\infty}(0,T;L^2(I))}\right).
\end{equation}
For $m = 0,1,\ldots,N$, denote by $\eta^{(m)}(x) = u(x,t_m) - u^{(m)}(x)$ the error of temporal discretization. The difference of \cref{trunc} and \cref{csPDE} yields the error equation 
\begin{equation}\label{erroreq}
\frac{\eta^{(m)}(x) - \eta^{(m-1)}(x)}{\Delta t_m}
= - \left(\mathcal{L}_1\left(u(x,t_m) \right) - \mathcal{L}_1\left(u^{(m)}(x) \right) \right) +\eta^{(m-1)}(x) + R^{(m)}(x) .
\end{equation}
For all $x\in I$, notice that 
\[
2\left(\eta^{(m)}-\eta^{(m-1)}\right)\eta^{(m)} = \Bigl(\eta^{(m)}\Bigr)^2 - \Bigl(\eta^{(m-1)}\Bigr)^2 + \left(\eta^{(m)}-\eta^{(m-1)}\right)^2 \ge \Bigl(\eta^{(m)}\Bigr)^2 - \Bigl(\eta^{(m-1)}\Bigr)^2
\]
and that
\begin{align}\notag
   \left(u(x,t_m)^3 - u^{(m)}(x)^3\right) \eta^{(m)}(x) & = \left(u(x,t_m)^2 +u(x,t_m)u^{(m)}(x) + u^{(m)}(x)^2 \right) \eta^{(m)}(x)^2 \\[2mm] \label{3}
   & \ge \dfrac12 \left(u(x,t_m)^2 + u^{(m)}(x)^2 \right) \eta^{(m)}(x)^2\ge 0.
\end{align}
Multiplying \cref{erroreq} by $\eta^{(m)}$, integrating by parts over $I$, and using the above two inequalities, we have
\begin{align*}\notag 
\frac{1}{2\Delta t_m}\Bigl(\bigl\|\eta^{(m)}\bigr\|^2 - \bigl\|\eta^{(m-1)}\bigr\|^2\Bigr)
& \leq -\varepsilon^2\bigl|\eta^{(m)}\bigr|_{H^1(I)}^2
+ \bigl(\eta^{(m-1)}, \eta^{(m)}\bigr) + \bigl(R^{(m)}, \eta^{(m)}\bigr) \\[2mm] 
& \leq \dfrac12 \|\eta^{(m-1)}\|^2 + \|\eta^{(m)}\|^2 + \dfrac12 \|R^{(m)}\|^2,
\end{align*}
which, together with the assumption that $\Delta t_m \leq 1/4$ and the fact that $\eta^{(0)}(x) = 0$, implies
\begin{align*}
    \bigl\|\eta^{(m)}\bigr\|^2
& \leq \dfrac{1 + \Delta t_m}{1-2\Delta t_m}\, \bigl\|\eta^{(m-1)}\bigr\|^2
+ \dfrac{\Delta t_m}{1-2\Delta t_m} \,\bigl\|R^{(m)}\bigr\|^2 
\leq \bigl(1 + 6\Delta t_m\bigr)\bigl\|\eta^{(m-1)}\bigr\|^2
+ 2\Delta t_m\,\bigl\|R^{(m)}\bigr\|^2
\\[2mm]
& \leq 2\sum_{k=0}^{m-1} \prod_{l=1}^{k} (1+6\Delta t_{m+1-l}) \Delta t_{m-k} \bigl\|R^{(m-k)}\bigr\|^2  \leq 2 \prod_{i=1}^{m}\bigl(1 + 6\Delta t_i\bigr) \sum_{k=1}^{m} \Delta t_{k}\bigl\|R^{(k)}\bigr\|^2 \\[2mm]
& \leq 2t_m\, \exp\left(6\sum\limits_{i=1}^{m}\Delta t_i\right) \max_{1\leq i \leq m}\bigl\|R^{(i)}\bigr\|^2 = 2 t_m e^{6 t_m} \max_{1\leq i \leq m}\bigl\|R^{(i)}\bigr\|^2.
\end{align*}
Now, \cref{temporalest} follows from the above inequality with $m=N$, the equality $t_N=T$, and the bound in \cref{remainderbound}. This completes the proof of the lemma.
\end{proof}

Next, for each time step we estimate the error between the solutions of the
minimization problems in \cref{umn} and \cref{minAC}. 

\begin{lemma}\label{onesteperror}
Let $u^{(m)}(x)$ and $u^{(m)}_n(x)$ be the solutions of \cref{umn} and \cref{minAC}, respectively. Let 
\[
\varrho_{n,m}= \inf_{v \in {\cal M}^{\alpha(t_m)}_n(I)} E(\bar{u}^{(m)}-v; \varepsilon^2, 1, 1/\Delta t_m),
\]
where $\bar{u}^{(m)} \in  {\cal V}_m$ is the minimizer of the functional $J_n^{(m)}(\cdot)$ over ${\cal V}_m$, i.e., 
\begin{equation}\label{u_bar_m}
 \bar{u}^{(m)}(x) = \argmin_{v \in {\cal V}_m} J^{(m)}_n(v).
 \end{equation}
Then for each $m = 1,\dots,N$, we have that
\begin{align}\label{onestepest}
\bigl\|u^{(m)} - u_n^{(m)}\bigr\|
\leq\; \bigl(1 + \Delta t_m\bigr)
\bigl\|u^{(m-1)} - u_n^{(m-1)}\bigr\|
 \;+\; \sqrt{\Delta t_m}\;  \varrho_{n,m} ,
\end{align}
and that
\begin{equation}\label{finalstepest}
\bigl\|u^{(N)} - u_n^{(N)}\bigr\|
\;\leq\; e^{\,T}\bigl\|u_0 - u_n^{(0)}\bigr\|
\;+\; \sum_{m=1}^{N} e^{\,T - t_m}\sqrt{\Delta t_m}\;\varrho_{n,m} .
\end{equation}
\end{lemma}

\begin{proof}
Similar to \cref{dr-tm}, $\bar{u}^{(m)}$ is the solution of \cref{pde} with 
\begin{equation*}
    a(x) = \varepsilon^2, \,\, r_1(x) = 1, \,\, r_2(x) = \tfrac{1}{\Delta t_m}, \,\mbox{ and }\,
f(x) = f^{(m)}_n(x) = \bigl(1 + \tfrac{1}{\Delta t_m}\bigr)u^{(m-1)}_n(x).
\end{equation*}
Hence, $\xi^{(m)} = u^{(m)} - \bar{u}^{(m)}$ satisfies 
\begin{equation*}\label{stabeq}
- \varepsilon^2 \bigl(\xi^{(m)}\bigr)_{xx}
+ \Bigl(\bigl(u^{(m)}\bigr)^3 - \bigl(\bar{u}^{(m)}\bigr)^3\Bigr)
+ \frac{1}{\Delta t_m}\xi^{(m)}
= \Bigl(1 + \frac{1}{\Delta t_m}\Bigr)\bigl(u^{(m-1)} - u_n^{(m-1)}\bigr)
\end{equation*}
with the homogeneous boundary conditions. Using its variational form with the test function $v=\xi^{(m)}$, an inequality similar to \cref{3}, and the Cauchy–Schwarz inequality, we have
\begin{align*}
\frac{1}{\Delta t_m}\bigl\|\xi^{(m)}\bigr\|^2 
& \leq \varepsilon^2\bigl|\xi^{(m)}\bigr|_{H^1(I)}^2
+ \Bigl(\bigl(u^{(m)}\bigr)^3 - \bigl(\bar{u}^{(m)}\bigr)^3,\, \xi^{(m)}\Bigr)
+ \frac{1}{\Delta t_m}\bigl\|\xi^{(m)}\bigr\|^2 \\[2mm]
&= \Bigl(1 + \frac{1}{\Delta t_m}\Bigr)
\bigl(u^{(m-1)} - u_n^{(m-1)},\, \xi^{(m)}\bigr) \leq \Bigl(1 + \frac{1}{\Delta t_m}\Bigr) \|u^{(m-1)} - u_n^{(m-1)}\|\, \bigl\|\xi^{(m)}\bigr\|,
\end{align*}
which implies
\begin{equation*}
\bigl\|\xi^{(m)}\bigr\|
\leq  \bigl(1 + \Delta t_m\bigr)\bigl\|u^{(m-1)} - u_n^{(m-1)}\bigr\| .
\end{equation*}
By the triangle inequality
\[
\bigl\|u^{(m)} - u_n^{(m)}\bigr\|
\leq \bigl\|\xi^{(m)}\bigr\| + \bigl\|\bar{u}^{(m)} - u_n^{(m)}\bigr\|, 
\]
the bound in \cref{onestepest} follows from \cref{l:error_estimate1} and 
\[
\bigl\|\bar{u}^{(m)} - u_n^{(m)}\bigr\| \leq \sqrt{\Delta t_m}\;\vertiii{\bar{u}^{(m)} - u_n^{(m)}}_{\varepsilon^2,1/\Delta t_m}
\leq \sqrt{\Delta t_m}\; \inf_{v \in {\cal M}^{\alpha(t_m)}_n(I)} E(\bar{u}^{(m)}-v; \varepsilon^2, 1, 1/\Delta t_m).
\]

Applying \cref{onestepest} successively for $m = 1,\dots,N$ and using
$u^{(0)} = u_0$ yields
\[
\bigl\|u^{(N)} - u_n^{(N)}\bigr\|
\leq \Biggl(\prod_{j=1}^{N}\bigl(1 + \Delta t_j\bigr)\Biggr)
\bigl\|u_0 - u_n^{(0)}\bigr\|
+ \sum_{m=1}^{N} \Biggl(\prod_{j=m+1}^{N}\bigl(1 + \Delta t_j\bigr)\Biggr)
\sqrt{\Delta t_m}\;\varrho_{n,m} ,
\]
where $\prod\limits_{j=N+1}^{N}\!\!\bigl(1 + \Delta t_j\bigr) = 1$ by convention. Now, \cref{finalstepest} is a direct consequence of the inequality $\prod\limits_{j=m+1}^{N}\!\!\bigl(1 + \Delta t_j\bigr)
\leq 
e^{T - t_m}$. This completes the proof of the lemma.
\end{proof}

The error estimate of the convex-splitting shallow Ritz method at the final time is given below.

\begin{theorem}\label{errorbound}
Under the assumptions of \cref{temporalerror}, we have
\begin{equation}\label{fullerrorest}
\bigl\|u(\cdot,T) - u_n^{(N)}\bigr\|
\;\leq\; C_T \Delta t
\;+\; e^{\,T}\bigl\|u_0 - u_n^{(0)}\bigr\|
\;+\; \sum_{m=1}^{N} e^{\,T - t_m}\sqrt{\Delta t_m}\;\varrho_{n,m} .
\end{equation}
\end{theorem}

\begin{proof} 
\cref{fullerrorest} is a direct consequence of the triangle inequality, \cref{temporalest}, and \cref{finalstepest}. 
\end{proof}

\begin{remark}
    Recall that the set ${\cal M}_n(I)$ is a class of approximating functions equivalent to the free-knot linear splines {\em (}see {\em \cite{Burchard74}} and the discussion in section~1 of {\em \cite{CaiDokFalHer2024a})}. From the definition in \cref{u_bar_m}, we obtain that $\bar{u}^{(m)} \in H^2(I)$. Therefore,
    \begin{equation*}
        \inf_{v \in {\cal M}^{\alpha(t_m)}_n(I)}\vertiii{\bar{u}^{(m)} - v}_{\varepsilon^2,1/\Delta t_m} \leq C(\bar{u}^{(m)})\left(\varepsilon + \frac{1}{\sqrt{\Delta t_m}}\right)\dfrac{1}{n}
    \end{equation*}
     for each $m = 1, \ldots, N$, and hence
    \begin{equation*}
        \varrho_{n,m} \leq \frac{C(\bar{u}^{(m)})}{n}\left(\varepsilon + \frac{1}{\sqrt{\Delta t_m}} + \bigl\|\bar{u}^{(m)}\bigr\|_{L^4(I)}\right).
    \end{equation*}
     Moreover, if $u_0$ has certain smoothness {\em (}see {\em \cite{DeVore1998Nonlinear, Schumaker})}, for example, $u_0 \in H^2(I)$, then for $u_n^{(0)}$ defined as in \cref{u_0LS}, we also have first-order convergence and therefore
    \begin{equation*}
        \bigl\|u(\cdot,T) - u_n^{(N)}\bigr\|
\;\leq\; C_T \Delta t
+ \frac{e^{\,T}C(u_0)}{n}
+ \frac{C_N}{n},
    \end{equation*}
    where $C_N = \sum\limits_{m=1}^{N} e^{\,T-t_m}\,C\left(\bar{u}^{(m)}\right)
\left(1 + \sqrt{\Delta t_m}\,\left(\varepsilon + \bigl\|\bar{u}^{(m)}\bigr\|_{L^4(I)}\right)\right)$.
\end{remark}



\section{Numerical Experiments} \label{Section numerical exp}

This section presents numerical results of the shallow Ritz and convex-splitting shallow Ritz methods for solving the semi-linear diffusion-reaction problem in \cref{pde} and the Allen–Cahn equation described in \cref{sec4}, respectively. The shallow Ritz approximations are obtained by using the rBN method when solving the respective \cref{min_pde} and \cref{minAC}. 

Both problems are non-convex optimization problems, and hence a good initialization is crucial. For the stationary problem in \cref{pde}, as in \cite{Liu_Cai_2022, CaiDokFalHer2024a}, the nonlinear parameters $\bb$, which represent the breakpoints partitioning the interval $I$, are initialized as the uniform partition of $I$. The linear parameters $\bc$ in problem \cref{min_pde} are then obtained by a Newton method until the residual is smaller than $10^{-12}$. This Newton method starts with an initial guess by solving a linearized system of \cref{linear_eq} described in \cref{linearization}, where $\hat{u}_n(x)$ is chosen as the neural network with $\bb$ corresponding to the uniform partition and $\bc = \mathbf{0}$.




For the time-dependent problem in \cref{minAC}, the least-squares NN approximation $u_n^{(0)}(x)$ to the initial data $u_0(x)$ in \cref{alg:BNAC} is computed by the rBN method with the stopping criteria: the residual falls below $10^{-6}$ or a maximum of $500$ iterations is reached. The NN approximation $u_{n}^{(m-1)}(x)$ at the previous time step serves as the initial guess at the current time $t_m$.

\subsection{Singularly Perturbed Diffusion-Reaction Equation}\label{sing_perturb}


The first test problem is the semi-linear diffusion-reaction equation in \cref{pde} with $a(x) = \varepsilon^2$, $r_1(x)=r_2(x)=1$, $I=(-1,1)$, and the exact solution
\[ 
    u(x) = \tanh{\left(\frac{x^2 - 1/4}{\varepsilon}\right)} - \tanh{\left(\frac{3}{4\varepsilon}\right)} + 1.
\] 

\begin{figure}[htb!]
    \centering
    \subfigure[]
    {\includegraphics[width=0.40\textwidth]{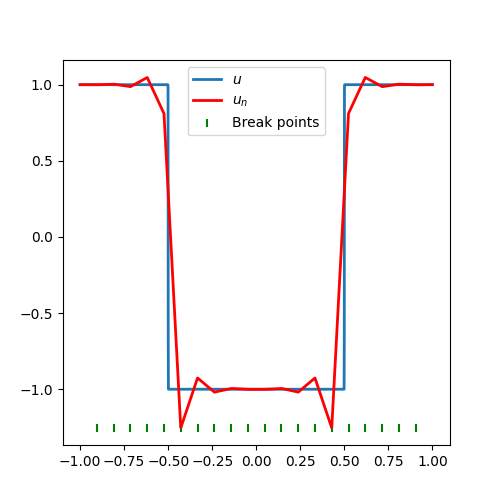}}
    \hspace{2em}
    \subfigure[]{
        \includegraphics[width=0.40\textwidth]{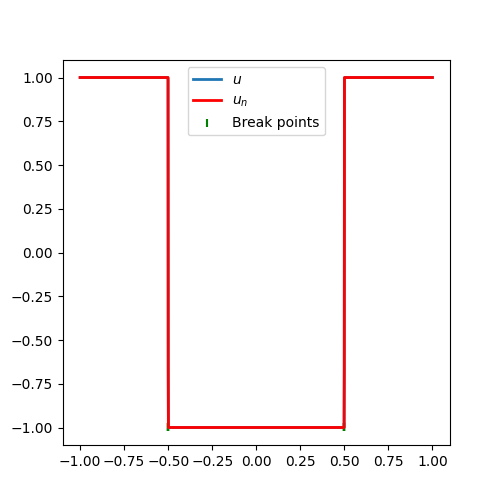}}
        \caption{(a) finite element and (b) shallow Ritz approximations for $a=10^{-7}$.} \label{example2DR}
\end{figure}

For small $a = \varepsilon^2$, the problem has sharp interior layers of width $\varepsilon$. Traditional mesh-based approaches such as the finite element method often result in overshooting and/or oscillations on relatively coarse uniform meshes as shown in \cref{example2DR}(a) for $a = 10^{-7}$ and $20$ uniform mesh points (the mesh size $h=2/21$). A shallow Ritz approximation using $20$ breakpoints is depicted in \cref{example2DR}(b); clearly, there is no overshooting and no oscillation. This approximation is obtained by $300$ rBN iterations with the initial breakpoints uniformly partitioning the interval $I$.

\cref{tab1:errorvsN} reports the relative errors in the $H^1$ seminorm $e_n = |u - u_n|_{H^1(I)}/|u|_{H^1(I)}$ 
for various values of $n$. For a relatively small increase in $n$, \cref{tab1:errorvsN} shows that the shallow Ritz method improves accuracy significantly while the improvement of the finite element approximation is negligible.

\begin{table}[h!]
    \begin{center}
        \begin{tabular}{ |p{2cm}||p{2cm}|p{2cm}| }
        \hline
            $n$&  $e_n$ (uniform)&  $e_n$ (moving) \\
            \hline
            20 & $9.95 \times10^{-1}$ &$1.15 \times10^{-1}$\\
            30 & $9.93 \times10^{-1}$ &$9.18 \times10^{-2}$ \\
            40 & $9.91 \times10^{-1}$ & $7.67 \times10^{-2}$ \\
            50 & $9.89 \times10^{-1}$ & $5.54 \times10^{-2}$\\
            60 &$9.86 \times10^{-1}$&$4.02 \times10^{-2}$\\

            \hline
        \end{tabular}
        \caption{Relative errors $e_n$ for $a = 10^{-7}$ with fixed and moving ($300$ rBN iterations) breakpoints.} 
        \label{tab1:errorvsN}
    \end{center}
\end{table}

For $n=20,\,40$ with fixed and moving breakpoints, \cref{comparison_epsilon} reports the relative errors in the $H^1$ seminorm for various values of $a=\varepsilon^2$. Again, the shallow Ritz approximations are obtained by using $300$ rBN iterations.
As expected, the finite element approximation deteriorates as the value of $a$ decreases. 
For each $a$ and $n$, the shallow Ritz method with rBN substantially improves the initial approximation, and the approximation accuracy is relatively robust with respect to $a$. Similar behavior was observed in \cite{CaiDokFalHer2024b} for the linear diffusion-reaction problem.

\begin{table}[htb!]
    \begin{center}
        \begin{tabular}{ |p{1cm}||p{2cm}|p{2cm}||p{2cm}|p{2cm}|}
        \hline
            & \multicolumn{2}{c||}{$n=20$} & \multicolumn{2}{c|}{$n=40$} \\
            \hline
            $a$ & $e_n$ (uniform) & $e_n$ (moving) & $e_n$ (uniform) & $e_n$ (moving) \\
            \hline
            $10^{-3}$ & $5.75\times10^{-1}$ & $2.31 \times10^{-1}$ & $3.78 \times 10^{-1}$ & $6.77 \times 10^{-2}$ \\ 
            $10^{-4}$ & $8.38\times10^{-1}$ & $1.37 \times10^{-1}$ & $7.00 \times 10^{-1}$ & $6.88 \times 10^{-2}$ \\ 
            $10^{-5}$ & $9.51\times10^{-1}$ & $1.33 \times10^{-1}$ & $9.02 \times 10^{-1}$ & $8.21 \times 10^{-2}$ \\
            $10^{-6}$ & $9.93\times 10^{-1}$ & $1.34 \times10^{-1}$ & $9.70 \times 10^{-1}$ &$6.41 \times 10^{-2}$ \\
            $10^{-7}$ & $9.95\times 10^{-1}$ & $1.15 \times10^{-1}$ & $9.91 \times 10^{-1}$ &$7.67 \times 10^{-2}$ \\
            \hline
        \end{tabular}
        \caption{Relative errors for $n=20,\,40$ and various $a$ with fixed and moving ($300$ rBN iterations) breakpoints.} 
        \label{comparison_epsilon}
    \end{center}
\end{table}

\subsection{Allen–Cahn Equation}

For the Allen–Cahn equation in \cref{sec4}, we present numerical results for two test problems. One has a manufactured solution with an interior layer, and the other has two interior layers with an unknown solution. 
Because the aim here is to evaluate the NN approximation within the spatial domain, a fixed time step of $\Delta t = 10^{-2}$ is used for both experiments. To measure errors, we use the relative discrete $\ell^2$ and $\ell^\infty$ errors at the final time $T$ defined respectively by  
\begin{equation*}
    e^{(2)}_n = \frac{\|u_n^{(N)}-u(\cdot, T)\|_{\ell^2}}{\|u(\cdot, T)\|_{\ell^2}} \quad \mbox{and}\quad
    e^{(\infty)}_n = \frac{\|u_n^{(N)}-u(\cdot, T)\|_{\ell^\infty}}{\|u(\cdot, T)\|_{\ell^\infty}}.
\end{equation*}


\subsubsection{One Interior Layer} 

The first test problem for the Allen–Cahn equation is a traveling wave solution 
\begin{equation*}
    u(x, t) = \frac{1}{2} - \frac{1}{2}\tanh\!\left(\frac{x - 0.5 - s t}{2\sqrt{2}\,\varepsilon}\right),
\end{equation*}
where $s = 3\varepsilon/\sqrt{2}$ denotes the wave speed and $\varepsilon = 10^{-2.5}$. \cref{tab2:errorvsN} reports the relative $H^1$ seminorm errors at time $T=1$ for various values of $n$ with fixed uniform and moving breakpoints. This test demonstrates that moving the mesh yields significantly more accurate approximations with fewer degrees of freedom than fixing the mesh.

\begin{table}[h!]
    \begin{center}
        \begin{tabular}{ |p{2cm}||p{2cm}|p{2cm}| }
        \hline
            $n$&  $e_n$ (uniform)&  $e_n$ (moving) \\
            \hline
            15 & $8.97 \times10^{-1}$ & $2.13 \times 10^{-1}$\\
            30 & $6.13 \times10^{-1}$ & $1.04 \times 10^{-1}$ \\
            45 & $5.22 \times10^{-1}$ & $8.54 \times 10^{-2}$ \\
            60 & $5.24 \times10^{-1}$ & $5.37 \times 10^{-2}$\\
            75 & $3.06 \times10^{-1}$ & $4.18 \times 10^{-2}$\\

            \hline
        \end{tabular}
        \caption{Relative errors $e_n$ for $\varepsilon = 10^{-2.5}$ with fixed uniform and moving breakpoints.}
        \label{tab2:errorvsN}
    \end{center}
\end{table}

\begin{figure}[htb!]
    \centering
    \subfigure[]{
        \includegraphics[width=0.40\textwidth]{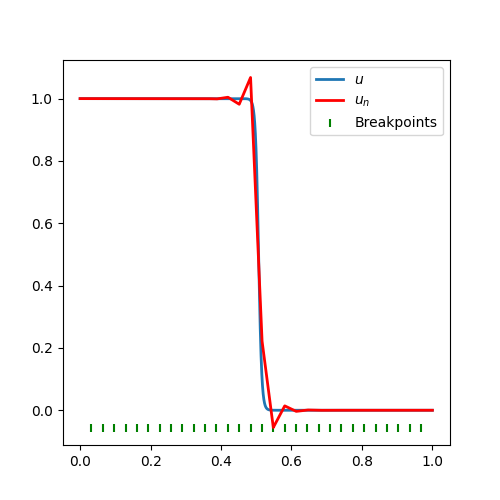}}
    \hspace{2em}
    \subfigure[]{
        \includegraphics[width=0.40\textwidth]{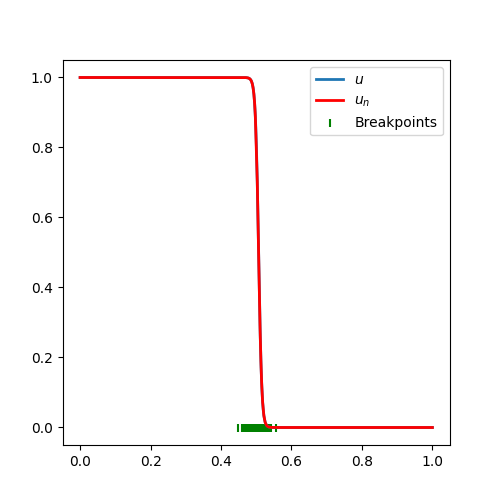}}
    \caption{For $\varepsilon= 10^{-2.5}$ and $T = 1$: (a) optimized NN model with 30 fixed uniform breakpoints, $e_n = 6.1\times 10^{-1}$; (b) optimized NN model with 30 moving breakpoints, $e_n =1.0\times 10^{-1}$.}
    \label{ex1AC_uniform}
\end{figure}

\begin{figure}[htb!]
    \centering
    \begin{minipage}[b]{0.32\textwidth}
        \centering
        \includegraphics[width=\textwidth]{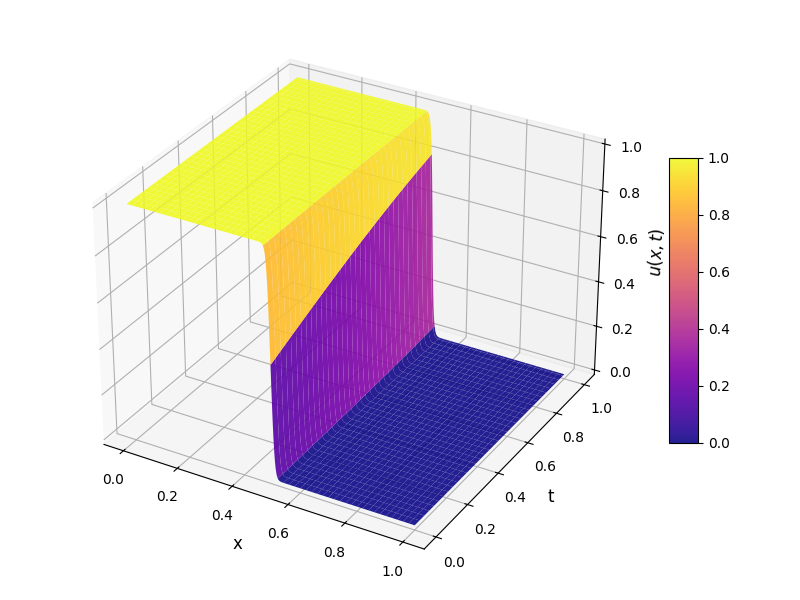}
        
        \small (a) Exact solution
    \end{minipage}
    \hfill
    \begin{minipage}[b]{0.32\textwidth}
        \centering
        \includegraphics[width=\textwidth]{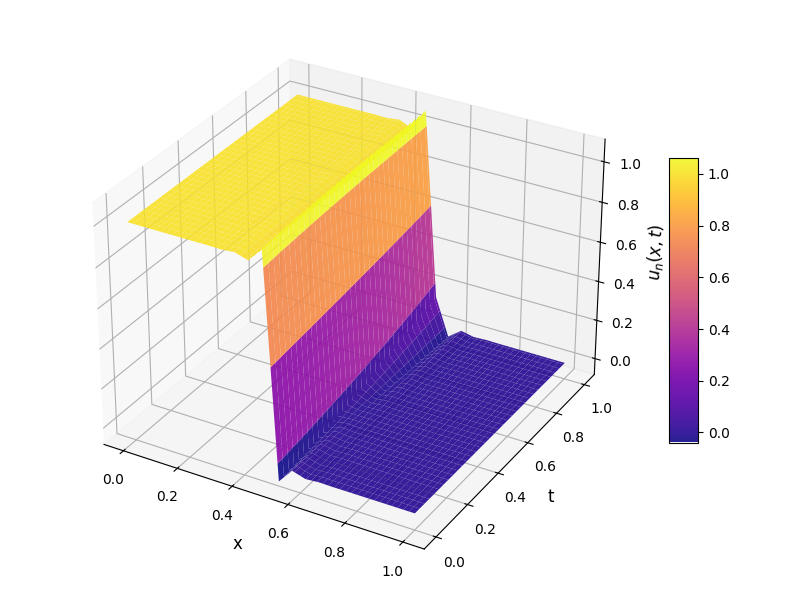}
        
        \small (b) Uniform mesh
    \end{minipage}
    \hfill
    \begin{minipage}[b]{0.32\textwidth}
        \centering
        \includegraphics[width=\textwidth]{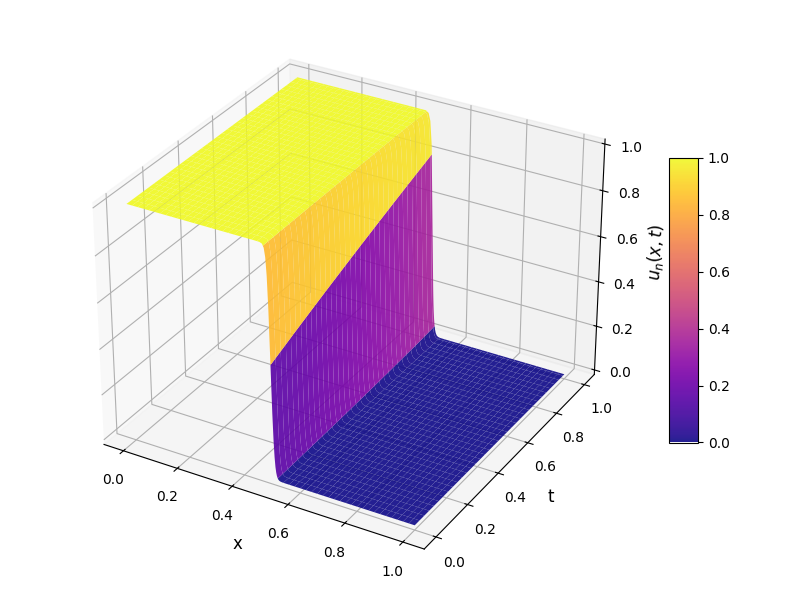}
        
        \small (c) Moving mesh
    \end{minipage}
    
    \vspace{1em}
    \caption{Neural network approximations on fixed and moving meshes for $\varepsilon = 10^{-2.5}$ and $n = 30$.}
    \label{ex_sp_AC}
\end{figure}

\cref{ex1AC_uniform}(a) depicts the convex-splitting linear finite element approximation with $n = 30$ uniform mesh points at time $T=1$, exhibiting the usual overshooting and oscillations. A significantly improved approximation by the convex-splitting shallow Ritz method is illustrated in \cref{ex1AC_uniform}(b), where $50$ rBN iterations are performed at each time step. \cref{ex_sp_AC} plots the exact solution and numerical approximations in space and time.


 

\subsubsection{Two Interior Layers}\label{secondTestAC}
The second problem for the Allen–Cahn equation has periodic boundary conditions and the following initial condition 
\begin{equation*}
    u(x, 0) =\tanh\left(\frac{x - \frac{1}{4}}{\sqrt{2}\varepsilon} \right) - \tanh\left(\frac{x - \frac{3}{4}}{\sqrt{2} \varepsilon}\right) -1 + 0.1\sin(2\pi x).
\end{equation*}
Since the solution at $x = 0$ and $x = 1$ for $t \in [0, 1.5]$ remains close to $-1$, the neural network approximation was obtained by replacing the periodic boundary condition with the Dirichlet condition $u(0, t) = u(1, t) = -1$. For comparison, a reference solution is computed by a Fourier spectral method on $10^4$ spatial points and a second-order BDF scheme \cite{shen2010numerical} with $10^{6}$ time steps for the respective spatial and temporal discretizations.


For $n = 25$, the least-squares NN approximations to $u(x,0)$ for $\varepsilon = 10^{-3}$ using uniform and moving meshes are shown in \cref{example2ACInit}, and the numerical solutions at $T = 1.5$ for $\varepsilon = 10^{-3.5}$ in \cref{example2ACmoving}. The space-time graphs of the reference and numerical solutions are depicted in \cref{ex2_sp_AC}. We observe that the moving mesh avoids the oscillations present in the uniform mesh approximation and yields a smaller relative error $e_n^{(\infty)}$.

\begin{figure}[htb!]
    \centering
    \subfigure[]
    {\includegraphics[width=0.40\textwidth]{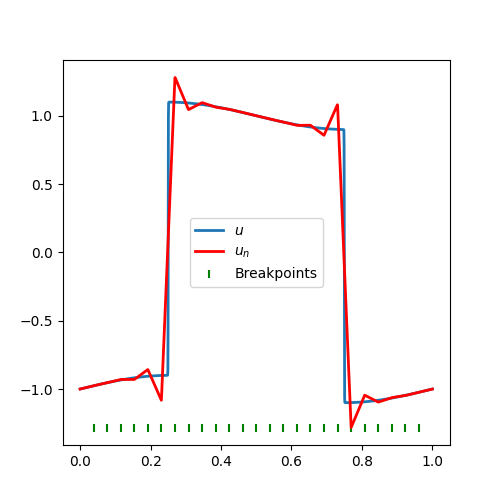}}
    \hspace{2em}
    \subfigure[]{
        \includegraphics[width=0.40\textwidth]{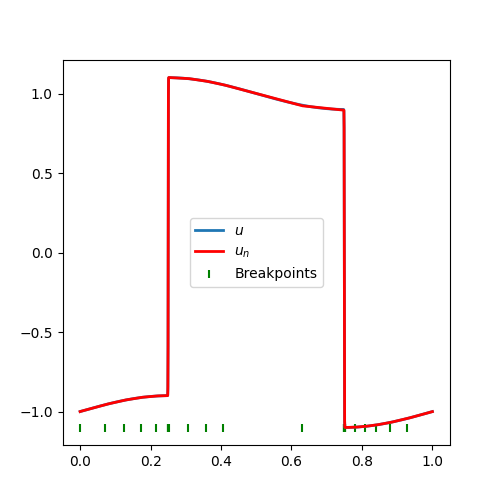}}
    \caption{For $\varepsilon = 10^{-3}$ and $25$ breakpoints, NN approximations at the initial time $t = 0$: (a) fixed, relative $L^2$ error $1.5\times 10^{-1}$; (b) moving, relative $L^2$ error $1.7\times 10^{-3}$.}
    \label{example2ACInit}
\end{figure}

\begin{figure}[htb!]
    \centering
    \subfigure[]
    {\includegraphics[width=0.40\textwidth]{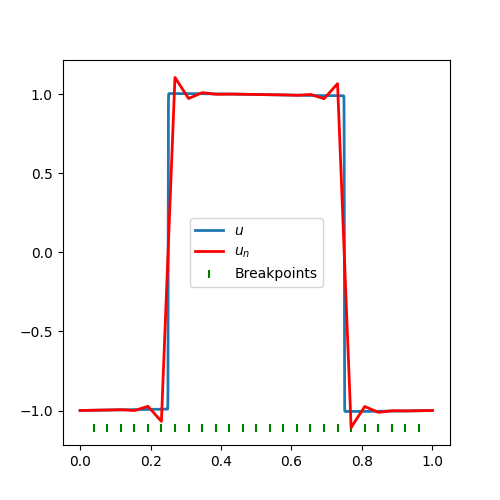}}
    \hspace{2em}
    \subfigure[]{
        \includegraphics[width=0.40\textwidth]{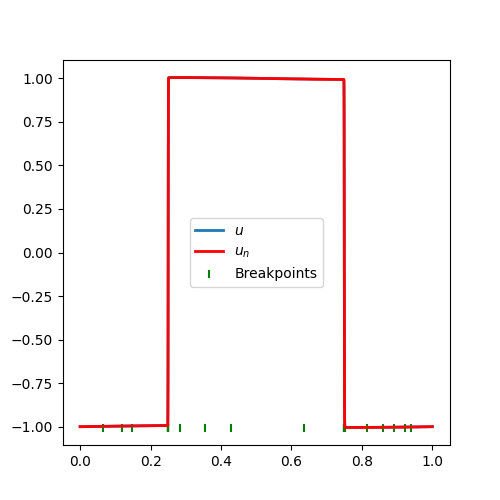}}
    \caption{For $\varepsilon  = 10^{-3.5}$ and $25$ breakpoints,  NN approximations at the final time $T = 1.5$: (a) fixed,  $e_n^{(\infty)} =9.3 \times 10^{-1}$; (b) moving, $e_n^{(\infty)} =3.47\times10^{-2}$.}
    \label{example2ACmoving}
\end{figure}

\begin{figure}[htb!]
    \centering
    \begin{minipage}[b]{0.32\textwidth}
        \centering
        \includegraphics[width=\textwidth]{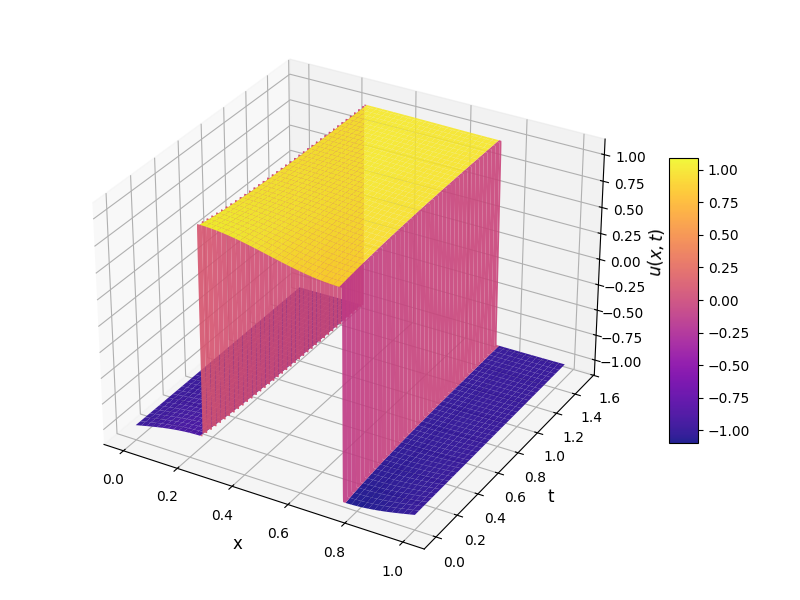}

\small (a) Reference solution
    \end{minipage}
    \hfill
    \begin{minipage}[b]{0.32\textwidth}
        \centering
        \includegraphics[width=\textwidth]{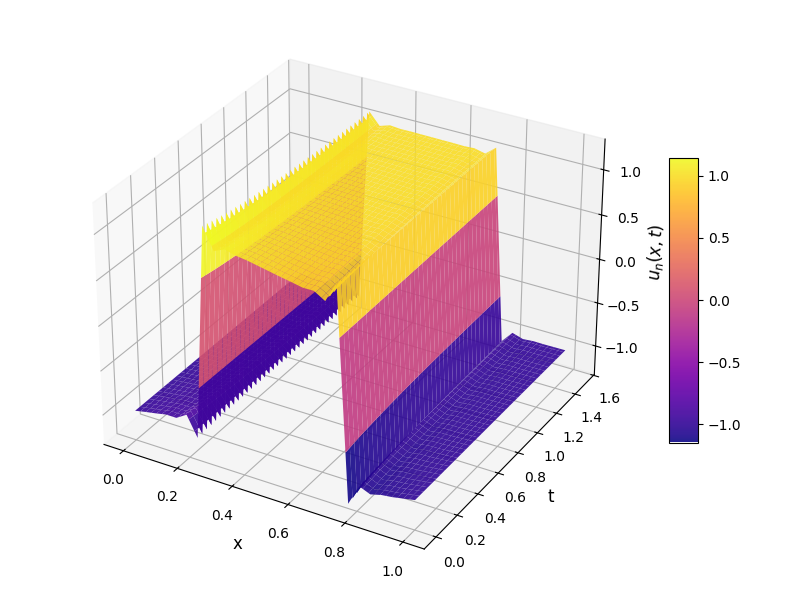}
        
        \small (b) Uniform mesh
    \end{minipage}
    \hfill
    \begin{minipage}[b]{0.32\textwidth}
        \centering
        \includegraphics[width=\textwidth]{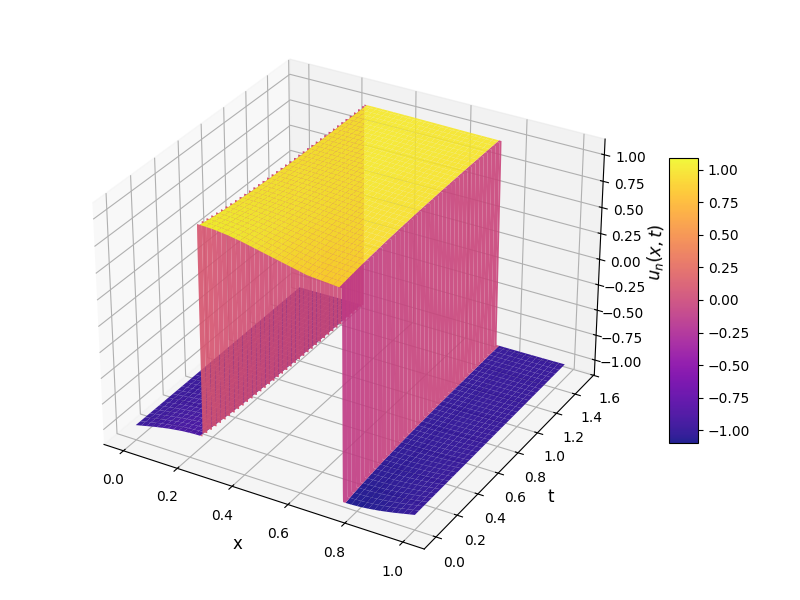}
        
        \small (c) Moving mesh
    \end{minipage}
    
    \vspace{1em}
    \caption{Neural network approximations on fixed and moving meshes for $\varepsilon = 10^{-3.5}$ and $n = 25$.}
    \label{ex2_sp_AC}
\end{figure}

Let $u^I(x)$ denote the piecewise linear interpolant of the reference solution $u(\cdot, T)$ at the final time $T$. The relative $H^1$ seminorm error with respect to this interpolant is defined by
\begin{equation*}
    e^{I}_n = \frac{|u_n^{(N)} - u^I|_{H^1(I)}}{|u^I|_{H^1(I)}}.
\end{equation*}  
At the final time $T=1.5$, \Cref{comparison_epsilon2} reports the NN approximation errors for different values of $\varepsilon$. The results show clear improvements with the moving mesh, while the error levels remain fairly consistent across the tested values of $\varepsilon$.

\begin{table}[htb!]
    \centering
    \begin{tabular}{ |p{1cm}||p{1.8cm}|p{1.8cm}|p{1.8cm}|p{1.8cm}|p{1.8cm}|p{1.8cm}| }
    \hline
    $\varepsilon$
& \multicolumn{1}{|p{1.8cm}|}{{\centering Uniform\arraybackslash}} 
      & \multicolumn{1}{p{1.8cm}|}{{\centering Moving\arraybackslash}} 
      & \multicolumn{1}{p{1.8cm}|}{{\centering Uniform\arraybackslash}} 
      & \multicolumn{1}{p{1.8cm}|}{{\centering Moving\arraybackslash}} 
      & \multicolumn{1}{p{1.8cm}|}{{\centering Uniform\arraybackslash}} 
      & \multicolumn{1}{p{1.8cm}|}{{\centering Moving\arraybackslash}} \\
    \hline
      & $e_n^{(2)}$ & $e_n^{(2)}$ & $e_n^{(\infty)}$ & $e_n^{(\infty)}$ & $e_n^{I}$ & $e_n^{I}$ \\
    \hline
    $10^{-1.5}$ & $1.7\times10^{-2}$ & $3.5 \times10^{-3}$ & $5.8 \times 10^{-2}$ & $1.3 \times 10^{-2}$ & $2.1 \times 10^{-1}$ & $9.0 \times 10^{-2}$ \\ 
    $10^{-2}$ & $3.2\times10^{-2}$ & $3.6 \times10^{-3}$ & $1.9 \times 10^{-1}$ & $2.7 \times 10^{-2}$ & $4.0 \times 10^{-1}$ & $1.2 \times 10^{-1}$ \\ 
    $10^{-2.5}$ & $1.0\times10^{-1}$ & $3.1 \times10^{-3}$ & $5.5 \times 10^{-1}$ & $4.6 \times 10^{-2}$ & $1.3$ & $1.5 \times 10^{-1}$ \\
    $10^{-3}$ & $1.4\times10^{-1}$ & $2.6 \times10^{-3}$ & $8.0 \times 10^{-1}$ & $5.5 \times 10^{-2}$ & $2.6$ & $1.5 \times 10^{-1}$ \\
    $10^{-3.5}$ & $1.5\times10^{-1}$ & $1.4 \times10^{-3}$ & $9.3 \times 10^{-1}$ & $3.4 \times 10^{-2}$ & $4.8$ & $1.7 \times 10^{-1}$ \\
    \hline
    \end{tabular}
    \caption{Relative errors for $n = 25$ and $\Delta t= 10^{-2}$ with uniform and moving breakpoints.}
    \label{comparison_epsilon2}
\end{table}


\section{Discussion and Conclusion}

The efficient shallow Ritz method developed in \cite{CaiDokFalHer2024a, CaiDokFalHer2024b} for linear diffusion-reaction problems was extended to a nonlinear problem with a cubic reaction term. A basic error estimate of the shallow Ritz method for this nonlinear problem was established in \cref{l:error_estimate1}. The lemma is similar to C\'{e}a's lemma and holds for multi-dimensional problems (see \cref{r1}).

A reduced block Newton (rBN) method was developed for the resulting finite-dimensional non-convex optimization problem. At each iteration, a nonlinear block Gauss–Seidel method was adopted based on blocks of the linear and nonlinear parameters, and each block was solved by Newton's method. A major contribution of rBN is to remove possible singularities of the corresponding Hessian matrix for the nonlinear parameter block when neurons do not contribute to the approximation or arrive at nearly optimal positions. Although the corresponding Hessian matrices can be inverted in $\mathcal{O}(n)$ operations in one dimension, an efficient solver is not yet available in multiple dimensions.


Combining the shallow Ritz method with a convex splitting scheme, we developed the convex-splitting shallow Ritz method for the Allen–Cahn equation. The method inherits the unconditional energy stability of the time discretization (see \cref{NNenergystabcor}). An error estimate at the final time $T$ in the $L^2$ norm was established in \cref{errorbound}. Numerical experiments confirm that the proposed approach produces stable and accurate solutions while significantly reducing errors compared to uniform mesh discretizations.


\section*{Acknowledgments}

The authors gratefully acknowledge Dr. Jie Shen for providing the code used to compute the reference solution for the test problem in \cref{secondTestAC}. The introduction was rewritten by ChatGPT based on a draft by the authors.

\smallskip


\appendix

\bibliographystyle{plain}
\bibliography{ref}

\clearpage

\end{document}